\documentclass[a4paper]{amsart}
\usepackage[utf8]{inputenc}
\usepackage[T1]{fontenc}

\usepackage{geometry}
\usepackage{amssymb}
\usepackage{amsfonts}
\usepackage{amsthm}
\usepackage{calrsfs}
\usepackage{array}
\usepackage{bbm}
\usepackage{stmaryrd}
\usepackage{hyperref}
\usepackage{enumerate}
\usepackage{mathabx}
\usepackage{enumitem}
\usepackage{graphicx}
\usepackage{caption}
\usepackage{subcaption}

\usepackage{csquotes}

\usepackage{mathtools}
\mathtoolsset{showonlyrefs=true}
\usepackage[svgnames]{xcolor}
\usepackage{hyperref}

\usepackage{tikz}
\usetikzlibrary{positioning}

\tikzset{
	dotA/.style={
		transform shape, fill,circle,inner sep=1.5pt, label distance=-1pt,
		font={\normalsize }
	},
	>=stealth,
}

\newcommand{\R}{\mathbb{R}}

\newcommand{\jap}[1]{\langle #1 \rangle}

\renewcommand{\Re}{\mathop{\text{Re}}}

\newcommand{\fia}{\mathbbm{1}_{|\Phi-\alpha|<M}}
\newcommand{\psia}{\mathbbm{1}_{|\Psi-\alpha|<M}}

\theoremstyle{plain}
\newtheorem{thm}{Theorem}[section]
\newtheorem*{thm*}{Theorem}
\newtheorem{prop}[thm]{Proposition}
\newtheorem{cor}[thm]{Corollary}
\newtheorem{lem}[thm]{Lemma}

\theoremstyle{definition}
\newtheorem{defi}[thm]{Definition}

\theoremstyle{remark}
\newtheorem{nb}[thm]{Remark}

\numberwithin{equation}{section}
\newtheoremstyle{mytheoremstyle} 
{\topsep}                    
{\topsep}                    
{}                   
{}                           
{\scshape}                   
{.}                          
{.5em}                       
{}  

\theoremstyle{mytheoremstyle} 
\theoremstyle{mytheoremstyle} 

\makeatletter
\@namedef{subjclassname@2020}{%
	\textup{2020} Mathematics Subject Classification}
\makeatother

\date{}
\author{Simão Correia, Felipe Linares and Jorge Drumond Silva}

\title{Sharp local well-posedness for the Schrödinger-Korteweg-de Vries system}

\keywords{Schrödinger-KdV system, local well-posedness, global existence.}

\subjclass[2020]{35A01, 35B65, 35Q53, 35Q55}

\thanks{S. C. and J.D.S. were partially supported by Funda\c{c}\~ao para a Ci\^encia e Tecnologia, through CAMGSD, IST-ID
	(projects UIDB/04459/2020 and UIDP/04459/2020) and through the project NoDES (PTDC/MAT-PUR/1788/2020).
	F.L. was partially supported by CNPq grant 310329/2023-0 and FAPERJ grant E-26/200.465/2023}

\allowdisplaybreaks

\begin{document}
\maketitle
\begin{abstract}
	We prove a sharp local existence result for the Schrödinger-Korteweg-de Vries system with initial data in $H^k(\R)\times H^s(\R)$. The proof is based on the concept of \textit{integrated-by-parts strong solution}, which generalizes the classical notion of strong solution, and on frequency-restricted estimates. Moreover, we extend the known global well-posedness result to regularities $k,s>1/2$.
\end{abstract}
\section{Introduction}

\subsection{Setting and main results}
In this work, we are interested in the well-posedness theory for the system
\begin{equation}\tag{NLS-KdV}\label{nlskdv}
	\begin{cases}
		iu_t + u_{xx} = \alpha uv + \beta |u|^2u\\
		v_t+v_{xxx} +\frac{1}{2}(v^2)_x = \gamma (|u|^2)_x
	\end{cases},\quad (t,x)\in \R \times \R, \quad \alpha, \beta, \gamma\in \R\setminus \{0\}.
\end{equation}

This system models interactions between short-wave ($u$) and long-wave ($v$), in fluid mechanics and plasma physics. From a merely mathematical
point of view, the first equation is a nonlinear Schrödinger type equation for the complex valued unknown $u$, whereas the second
equation is a KdV type one for real valued $v$, with coupling terms between the two with parameters $\alpha$ and $\gamma$. The case $\beta=0$
has particular significance, as it occurs when modeling the resonant interaction between short and long capillary-gravity
waves on water of uniform finite depth, in plasma physics and in a diatomic
lattice system (see \cite{corcholinares} and references therein for more background on this system).

\medskip
Concerning the local well-posedness theory for the Cauchy problem with initial data $(u_0, v_0) \in H^k(\R)\times H^s(\R)$, the first result was obtained by Corcho and the second author \cite{corcholinares} for exponents satisfying
$$k \ge 0,\quad s> -\frac{3}{4},\quad 
\begin{cases}
	k-1\le s \le 2k-\frac{1}{2},\quad \mbox{for }0\le k \le \frac{1}{2}\\
	k-1\le s \le k+\frac{1}{2},\quad \mbox{for }k>\frac{1}{2}
\end{cases}.
$$
Their proof is based on the Fourier restriction norm (or Bourgain space, see \eqref{eq:defi_bourg}) method \cite{bourg1, bourg2}. Later, Wu \cite{wu} improved the multilinear estimates derived in \cite{corcholinares} and extended the local well-posedness theory to regularities in
the planar region of exponent pairs $(k,s)$ given by
$$
\mathcal{A}_0:=\left\{ (k,s)\in \R^2 : k\ge 0,\quad s>-\frac{3}{4}, \quad s<4k, \quad -1<k-s<2\right\}.
$$
Moreover, he showed that the first three restrictions are sharp for the existence of a (at least) $C^2$ data-to-solution flow in $H^k(\R)\times H^s(\R)$. We also refer to \cite{guowang}, where the local existence was shown for the endpoint case $(k,s)=(0,-3/4)$ (see also \cite{wangcui}). Previous well-posedness results can also be found in the works of Bekiranov, Ogawa and Ponce \cite{BOP97} and Tsutsumi \cite{tsutsumi}.

The obstructions $-1<k-s<2$ found in \cite{wu} are due to the fact that the bilinear estimates in Bourgain spaces, crucial for the employment of this method,  fail outside of this strip. This, however, is not really due to a true obstruction in the flow, as one can see by plugging linear solutions in the coupling terms, but possibly a consequence of the particular choice of functional space in which one seeks to prove the well-posedness result by contraction through a Picard iteration scheme. 

The problem of bypassing possibly artificial obstructions caused by the framework of obtaining solutions as fixed points of the Duhamel formula in certain functional spaces has been considered for other dispersive equations. In \cite{taobejenaru}, the authors consider a quadratic Schrödinger equation, for which the straightforward  use of standard Bourgain spaces yields local well-posedness for $s>-3/4$, quite far from the scaling $s_c=-3/2$. The restriction $s>-3/4$ appears in a highly time-oscillatory region, on which linear solutions are actually well-behaved. The idea was to adapt the Bourgain space in order to capture more precisely the nonlinear interaction of near-linear solutions. This led to the threshold $s\ge -1$, which is sharp in the sense of $C^2$-flows. Another example is the KP-II equation in anisotropic Sobolev spaces $H^{s,0}(\R^2)$, where the regularity threshold was successively improved  by Bourgain \cite{bourg_kpii} ($s=0$), Takaoka and Tzvetkov \cite{takaokatzvetkov} ($s=-1/3$) and finally Hadac \cite{hadac}, down to the scaling-regularity $s=-1/2$, by means of suitable tweaking of the definition of the Fourier restriction norm space.

While faced with the same problem, we present here a different approach, which does not rely on any modification of the functional space, but rather exploits more thoroughly the time-oscillations observed in the problematic region. This is achieved by introducing the concept of \textit{integrated-by-parts strong solution}, which can be seen as a generalization of the usual concept of strong integral solution (i.e., a solution to Duhamel's integral formula in the given functional space). 
As a result, we will complete the local well-posedness theory for \eqref{nlskdv} for initial data in $H^k(\R)\times H^s(\R)$.
Define the admissible regularity region
$$
\mathcal{A}:=\left\{ (k,s)\in \R^2 : k\ge 0,\quad s>-\frac{3}{4}, \quad s<4k, \quad -2<k-s<3\right\}.
$$

The first result can then be roughly stated as
\begin{thm}\label{thm:well}
	System \eqref{nlskdv} is locally well-posed in $H^k(\R)\times H^s(\R)$ for $(k,s)\in \mathcal{A}$.
\end{thm}
For a more precise statement, see Theorem \ref{thm:well2} below.

\begin{nb}
	The exploitation of oscillations in time through an integration-by-parts has appeared in other contexts under the names of \textit{differentiation-by-parts} or \textit{normal form reduction}, see for example \cite{bit}, \cite{GMS},\cite{kishimoto} and \cite{OS12}.
\end{nb}

The main argument of the proof of Theorem \ref{thm:well} consists of two steps. First, we construct the solution for some lower regularity near the diagonal case $k=s$ (more precisely, in region $\mathcal{A}_0$). Then, via the nonlinear smoothing effect achieved through the integration-by-parts, we recover the persistence of the solution at the higher level of regularity of the initial data.

As we derive the required multilinear bounds necessary to prove Theorem \ref{thm:well}, we observe that the restrictions $-2<k-s<3$ are caused exactly by the interaction of linear solutions in the nonlinear terms (this can be seen in the estimates for the time-boundary terms, see Lemma \ref{lem:bdryHs}). Consequently, we are able to pinpoint the exact obstructions to a local well-posedness result away from the regularity region $\mathcal{A}$:
\begin{thm}\label{thm:illposed}
	System \eqref{nlskdv} is $C^2$-ill-posed in $H^k(\R)\times H^s(\R)$ for $(k,s)\notin \overline{\mathcal{A}}$.
\end{thm}

By $C^2$-ill-posedness we mean, of course, that the data-to-solution map from $H^k(\R)\times H^s(\R)$ to $C_tH^k(\R)\times C_tH^s(\R)$, $(u_0,v_0)\mapsto (u,v),$ fails to be $C^2$.

\begin{figure}[h]
	\includegraphics[width=6cm]{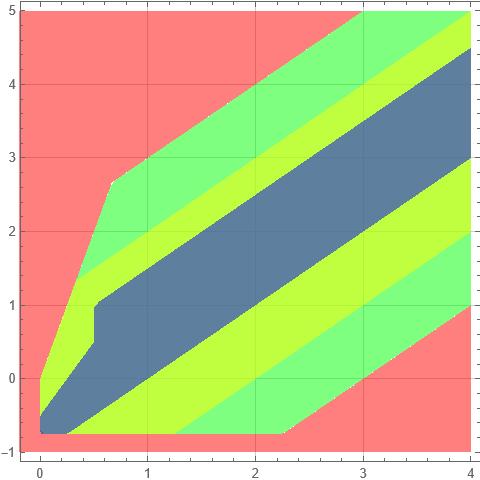}
	\caption{The $(k,s)$ regularity plane. One has ill-posedness in the red region (\cite{wu} and Theorem \ref{thm:illposed}) and local well-posedness in the blue \cite{corcholinares}, yellow \cite{wu} and green regions (Theorem \ref{thm:well})}
\end{figure}

\medskip

Finally, we turn to the question of global well-posedness.
In \cite{corcholinares}, Corcho and the second author proved the global well-posedness of \eqref{nlskdv} in $H^1(\R)\times H^1(\R)$ using the mass and energy conservation laws. Pecher \cite{pecher} applied the I-method to prove global well-posedness in $H^k(\R)\times H^s(\R)$ for $k=s\ge 3/5$. This restriction was improved by Wu \cite{wu} by refining the estimate for the local time of existence, reaching the current state-of-the-art $k=s>1/2$. 

A common feature of these global results is the restriction $k=s$. This is due to the fact that persistence of regularity is only known when both $k$ and $s$ are incremented by the same amount. Indeed, as one attempts to control higher-order derivatives of $u$ through the equation, these extra derivatives also fall onto $v$ (and vice-versa). In other words, the persistence of regularity is only valid for equal increments in both $k$ and $s$. As such, a global result at $k_0=s_0$ can only be extended to $k=s>s_0$.

In order to bypass this problem, one must control higher-order derivatives of the nonlinear terms without distributing the extra derivatives onto the various factors. This means that the structure of the nonlinearity must itself be capable of absorbing these derivatives, leaving the factors in the original regularity level. This is once again tantamount to proving a nonlinear smoothing effect for the coupling terms. Fortunately, our approach to proving local well-posedness already exploits the nonlinear smoothing effects and the persistence of regularity away from the diagonal $k=s$ will come almost for free, yielding the following result.

\begin{thm}\label{thm:gwp}
System \eqref{nlskdv} is globally well-posed in $H^k(\R)\times H^s(\R)$ for $(k,s)\in \mathcal{A}$, $k,s> 1/2$.
\end{thm}

\medskip

\begin{nb}
	The arguments employed in this work are not intrinsic to the \eqref{nlskdv} system and should be applicable to other systems of coupled dispersive equations. For example, \cite{domingues} considers the Schrödinger-Benjamin-Ono system
	$$
	\begin{cases}
		iu_t + u_{xx} = \alpha uv\\
		v_t + \nu \mathcal{H}v_{xx} = \beta (|u|^2)_x
	\end{cases},\quad \mathcal{H} = \mbox{ Hilbert transform,}
	$$
	while \cite{barbosa} focuses on the quadratic Schrödinger system
		$$
	\begin{cases}
		iu_t + u_{xx} + \bar{u}v=0\\
		i\sigma v_t + v_{xx} +\frac{1}{2\sigma} u^2 =0
	\end{cases},\quad \sigma\in \R.
	$$
	The local well-posedness results obtained therein include a gap in regularities $(k,s)$ similar to that of \cite{wu} (see, for example, Figure 1 in \cite{domingues}). It would be interesting to see if such gaps can be closed through our technique or if more serious obstructions are in effect.
\end{nb}

\medskip

The paper is structured as follows. In Section \ref{sec:method}, we introduce the functional framework and the concept of integrated-by-parts strong solution. In Section \ref{sec:fre}, we reduce the required estimates to frequency-restricted estimates (or FRE). Section \ref{sec:multi} is devoted to the proof of the various estimates required for the local well-posedness theory. In Section \ref{sec:well}, we prove the local and global results, while the proof of the ill-posedness results is done in Section \ref{sec:ill}.

\bigskip

\noindent\textit{Notation.} Throughout this work, given $a\in \R$, $a^+$ (resp. $a^-$) denotes a number slightly larger (resp. smaller) than $a$. Given $a,b\in \R$, $a\lesssim b$ means that there exists a universal constant $C>0$ such that $a\le Cb$. If $C$ can be chosen to be small, we write $a\ll b$. If $a\lesssim b$ and $b\lesssim a$, we write $a\sim b$. Moreover, if $|b-a|\ll |a|$, we say that $a\simeq b$. Finally, $\mathcal{F}_\ast$ denotes the Fourier transform in the $\ast$ variable. If the Fourier transform is taken in all variables, we may also denote it as $\hat{\cdot}$.

\section{Description of the method}\label{sec:method}

We define the Schrödinger Bourgain space $X^{s,b}$ in the usual way, as the closure of $\mathcal{S}(\R \times \R)$ induced by the norm
\begin{equation}\label{eq:defi_bourg}
	\|u\|_{X^{s,b}}:= \left( \int \jap{\tau+\xi^2}^{2b}\jap{\xi}^{2s}|\hat{u}(\tau,\xi)|^2 d\tau d\xi\right)^{1/2},
\end{equation}
where $\hat{u}(\tau,\xi)$ here denotes the Fourier transform of $u$ in both the time and space variables.
Analogously, the Airy Bourgain space $Y^{s,b}$ is defined via the norm
\begin{equation}\label{eq:defi_bourg2}
\|v\|_{Y^{s,b}}:= \left( \int \jap{\tau-\xi^3}^{2b}\jap{\xi}^{2s}|\hat{v}(\tau,\xi)|^2 d\tau d\xi\right)^{1/2}.
\end{equation}

We will apply a Picard iteration scheme to the Duhamel formulation for \eqref{nlskdv} seeking fixed points in these 
Bourgain spaces. As is well-known, this method hinges on multilinear estimates for 
the nonlinear terms of the equations. We begin by recalling the classic multilinear estimates for the cubic (NLS) and (KdV).

\begin{lem}[\cite{kpv_bilin}]\label{lem:est_kdv}
	For $s>-3/4$, $b=\frac{1}{2}^+$ and $b'=(b-1)^+$,
	$$
	\|\partial_x(v^2)\|_{Y^{s,b'}}\lesssim \|v\|_{Y^{s,b}}^2.
	$$
\end{lem}
\begin{lem}[\cite{BOP98}]\label{lem:est_nls}
	For $k\ge 0$, $b=\frac{1}{2}^+$ and $b'=(b-1)^+$,
	$$
	\||u|^2u\|_{X^{k,b'}}\lesssim \|u\|_{X^{k,b}}^3.
	$$
\end{lem}
These estimates are enough to control the corresponding terms in system \eqref{nlskdv} and we are left with the derivation of sharp regularity estimates for the coupling terms. 

\begin{lem}[\cite{wu}]\label{lem:est_wu}
	Given $b=\frac{1}{2}^+$ and $b'=(b-1)^+$, the following inequalities hold
\begin{equation}\label{eq:estwu1}
		\|uv\|_{X^{k,b'}}\lesssim \|u\|_{X^{k,b}}\|v\|_{Y^{s,b}},\quad \mbox{ for } s>-1,\  k-s<2, 
\end{equation}
and
\begin{equation}\label{eq:estwu2}
		\|\partial_x(u_1\overline{u_2})\|_{Y^{s,b'}}\lesssim \|u_1\|_{X^{k,b}}\|u_2\|_{X^{k,b}},\quad \mbox{ for }s<4k,\  s-k<1.
\end{equation}
\end{lem}
These estimates are optimal for $b=1/2^+$ (see Corollaries \ref{cor:counter} and \ref{cor:counter2}).
In order to extend the admissible regularity region considered in \cite{wu} to $\mathcal{A}$, one must first understand precisely the frequency interactions responsible for the obstructions.

First, in the bilinear estimate
	$$
\|uv\|_{X^{k,b'}}\lesssim \|u\|_{X^{k,b}}\|v\|_{Y^{s,b}},
$$
the obstruction $k-s<2$ is caused specifically by the spatial frequency interaction $u_{\text{low}} \times v_{\text{high}}$  (see Lemma \ref{lem:probU}). In this frequency domain, the resonance function, $\Phi_1^u=\xi^2-\xi_1^2+\xi_2^3$, is comparable to $\xi^3$, inducing a strong oscillatory effect in the nonlinear term.

On the other hand, in the estimate
\begin{equation}
\|\partial_x(|u|^2)\|_{Y^{s,b'}}\lesssim \|u\|_{X^{k,b}}^2,
\end{equation}
the restriction $k-s>-1$ stems from the spatial frequency interaction $u_{\text{low}} \times u_{\text{high}}$ (see Lemma \ref{lem:probV}). For this interaction, the corresponding resonance function $\Phi^v_1=-\xi^3-\xi_1^2+\xi_2^2$ is once again comparable to $\xi^3$.

The main idea for our improvement on the local well-posedness is to explore more deeply the strong oscillation in time in these problematic regions. Given $\delta^u, \delta^v>0$, define sets
$$
U_\xi=\{\xi_1\in \R: 100|\xi_1|<|\xi| \text{ and }1/\delta^u<|\xi|\},\quad V_\xi=\left\{ \xi_1\in \R: 1/\delta^v<|\xi_1|< 100|\xi| \right\}
$$
and the profiles
$$
\tilde{u}(t)=e^{it\xi^2}\mathcal{F}_xu(t),\quad \tilde{v}(t)=e^{-it\xi^3}\mathcal{F}_xv(t),
$$
where
 $$\mathcal{F}_xu(t,\xi)=\int e^{-i\xi x}u(t,x)dx,$$
 denotes the Fourier transform of $u$ in the spatial variable $x$ only, and analogously for $v$.
In what follows, we abbreviate $\tilde{u}(s,\xi_k)$ as $\tilde{u}_k$, for any index $k$, and write $\tilde{u}^\ast(\xi)=\overline{\tilde{u}(-\xi)}$. In terms of these profiles, the Duhamel formulas for the \eqref{nlskdv} system can be written as
\begin{align}
	\tilde{u}(t,\xi)&=\tilde{u}(0,\xi)\label{eq:perfilu} -i\alpha\int_0^t\int_{\xi=\xi_1+\xi_2} e^{is\Phi_1^u}\tilde{u}_1\tilde{v}_2 d\xi_1 ds  -i\beta \int_0^t\int_{\xi=\xi_1+\xi_2+\xi_3} e^{is\Phi_2^u}\tilde{u}_1\tilde{u}_2^\ast\tilde{u}_3 d\xi_1d\xi_2ds\\
		\tilde{v}(t,\xi)&=\tilde{v}(0,\xi) +i\gamma \int_0^t\int_{\xi=\xi_1+\xi_2}\xi\label{eq:perfilv} e^{is\Phi_1^v}\tilde{u}_1\tilde{u}_2^\ast d\xi_1 ds  -i\frac12 \int_0^t\int_{\xi=\xi_1+\xi_2} e^{is\Phi_2^v}\xi \tilde{v}_1\tilde{v}_2 d\xi_1ds
\end{align}
where the resonance functions are given by
$$
\Phi^u_1=\xi^2-\xi_1^2+\xi_2^3,\quad \Phi^u_2=\xi^2-\xi_1^2+\xi_2^2-\xi_3^2
$$
and
$$
\Phi^v_1=-\xi^3-\xi_1^2+\xi_2^2,\quad \Phi^v_2=-\xi^3+\xi_1^3+\xi_2^3.
$$
Consider the coupling term in \eqref{eq:perfilu}. We split the integration in the regions $U_\xi$ and $U^c_\xi$ and integrate by parts in time in the first:
\begin{align*}
\int_0^t\int_{\xi=\xi_1+\xi_2} e^{is\Phi_1^u}\tilde{u}_1\tilde{v}_2 d\xi_1 ds &= \left[\int_{U_\xi}\frac{e^{is\Phi^u_1}}{i\Phi^u_1}\tilde{u}_1\tilde{v}_2d\xi_1\right]_{s=0}^{s=t} - \int_0^t\int_{U_\xi} \frac{e^{is\Phi_1^u}}{i\Phi_1^u}\partial_s(\tilde{u}_1\tilde{v}_2) d\xi_1 ds \\&\qquad+\int_0^t\int_{U_\xi^c} e^{is\Phi_1^u}\tilde{u}_1\tilde{v}_2 d\xi_1 ds\\&= [B^u[u,v](s,\xi)]_{s=0}^{s=t} - \int_0^t\int_{U_\xi} \frac{e^{is\Phi_1^u}}{i\Phi_1^u}\partial_s(\tilde{u}_1\tilde{v}_2) d\xi_1 ds +\int_0^t N_0^u[u,v](s,\xi)ds,
\end{align*}
where $$N_0^u[u,v](s,\xi)=\int_{U_\xi^c} e^{is\Phi_1^u}\tilde{u}_1\tilde{v}_2 d\xi_1.$$ 

We now replace the term $\partial_s(\tilde{u}_1\tilde{v}_2)$ using \eqref{eq:perfilu} and \eqref{eq:perfilv}:
\begin{align*}
-\int_0^t\int_{U_\xi} \frac{e^{is\Phi_1^u}}{i\Phi_1^u}\partial_s(\tilde{u}_1\tilde{v}_2) d\xi_1 ds &=i\alpha \int_0^t \int_{\xi_1\in U_\xi,\, \xi_{1}=\xi_{11}+\xi_{12}} \frac{e^{is\Psi_1^u}}{i\Phi_1^u}\tilde{u}_{11}\tilde{v}_{12}\tilde{v}_2 d\xi_{11} d\xi_{12} ds \\&+i\beta \int_0^t \int_{\xi_1\in U_\xi,\, \xi_{1}=\xi_{11}+\xi_{12}+\xi_{13}} \frac{e^{is\Psi_2^u}}{i\Phi_1^u}\tilde{u}_{11}\tilde{u}_{12}^\ast\tilde{u}_{13}\tilde{v}_2 d\xi_{11} d\xi_{12} d\xi_{13}ds \\ &-i\gamma\int_0^t \int_{\xi_1\in U_\xi,\, \xi_{2}=\xi_{21}+\xi_{22}} \frac{e^{is\Psi_3^u}}{i\Phi_1^u}\xi_2\tilde{u}_{1}\tilde{u}_{21}\tilde{u}_{22}^\ast d\xi_{21} d\xi_{22} ds\\ &+i\frac12 \int_0^t \int_{\xi_1\in U_\xi,\, \xi_{2}=\xi_{21}+\xi_{22}} \frac{e^{is\Psi_4^u}}{i\Phi_1^u}\xi_2\tilde{u}_{1}\tilde{v}_{21}\tilde{v}_{22} d\xi_{21} d\xi_{22} ds\\&= \int_0^t\left(N_1^u[u,v,v]+N_2^u[u,u,u,v]+N_3^u[u,u,u]+N_4^u[u,v,v]\right)(s,\xi)ds.
\end{align*}
where the new resonance functions are given by
$$
\Psi_1^u= \xi^2-\xi_{11}^2+\xi_{12}^3+\xi_2^3,\quad \Psi_2^u= \xi^2-\xi_{11}^2+\xi_{12}^2-\xi_{13}^2+\xi_2^3, 
$$
$$
\Psi_3^u= \xi^2-\xi_{1}^2-\xi_{21}^2+\xi_{22}^2,\quad \Psi_4^u= \xi^2-\xi_{1}^2+\xi_{21}^3+\xi_{22}^3. 
$$
In the same way, we split the coupling term in \eqref{eq:perfilv} into the domains $V_\xi$ and $V^c_\xi$, integrate by parts in the first and use once again \eqref{eq:perfilu}-\eqref{eq:perfilv} to replace time derivatives:

\begin{align*}
	\int_0^t\int_{\xi=\xi_1+\xi_2}\xi e^{is\Phi_1^v}\tilde{u}_1\tilde{u}_2^\ast d\xi_1 ds &= \left[ \int_{\xi=\xi_1+\xi_2}\xi \frac{e^{is\Phi_1^v}}{i\Phi_1^v}\tilde{u}_1\tilde{u}_2^\ast d\xi_1 \right]_{s=0}^{s=t} - \int_0^t\int_{V_\xi}\xi \frac{e^{is\Phi_1^v}}{i\Phi_1^v}\partial_s(\tilde{u}_1\tilde{u}_2^\ast) d\xi_1 ds  \\&+ \int_0^t\int_{V_\xi^c}\xi e^{is\Phi_1^v}\tilde{u}_1\tilde{u}_2^\ast d\xi_1 ds \\&= \left[ \int_{\xi=\xi_1+\xi_2}\xi \frac{e^{is\Phi_1^v}}{i\Phi_1^v}\tilde{u}_1\tilde{u}_2^\ast d\xi_1 \right]_{s=0}^{s=t} + \int_0^t\int_{V_\xi^c}\xi e^{is\Phi_1^v}\tilde{u}_1\tilde{u}_2^\ast d\xi_1 ds\\ &+i\alpha \int_0^t\int_{\xi_1\in V_\xi,\, \xi_1=\xi_{11}+\xi_{12}}\xi \frac{e^{is\Psi_1^v}}{i\Phi_1^v}\tilde{u}_{11}\tilde{v}_{12}\tilde{u}_2^\ast d\xi_{11}d\xi_{12} ds
	\\&+i\beta \int_0^t\int_{\xi_1\in V_\xi,\, \xi_1=\xi_{11}+\xi_{12}+\xi_{13}}\xi \frac{e^{is\Psi_2^v}}{i\Phi_1^v}\tilde{u}_{11}\tilde{u}_{12}^\ast\tilde{u}_{13}\tilde{u}_2^\ast d\xi_{11}d\xi_{12}d\xi_{13} ds \\ &-i \alpha \int_0^t\int_{\xi_1\in V_\xi,\, \xi_2=\xi_{21}+\xi_{22}}\xi \frac{e^{is\Psi_3^v}}{i\Phi_1^v}\tilde{u}_1\tilde{u}_{21}^\ast \tilde{v}^\ast_{22} d\xi_{21}d\xi_{22} ds
	\\&- i\beta\int_0^t\int_{\xi_1\in V_\xi,\, \xi_2=\xi_{21}+\xi_{22}+\xi_{23}}\xi \frac{e^{is\Psi_4^v}}{i\Phi_1^v}\tilde{u}_{1}\tilde{u}^\ast_{21}\tilde{u}_{22}\tilde{u}^\ast_{23} d\xi_{11}d\xi_{12}d\xi_{13} ds \\&=\left[ B^v[u,u](s,\xi) \right]_{s=0}^{s=t} + \int_0^tN_0^v[u,u](s,\xi)ds \\&+\int_0^t  \left(N^v_1[u,v,u]+N^v_2[u,u,u,u]+N^v_3[u,u,v]+N^v_4[u,u,u,u]\right)(s,\xi)ds
\end{align*}
with resonance functions
$$
\Psi^v_1=-\xi^3-\xi_{11}^2+\xi_{12}^3+\xi_2^2,\quad \Psi_2^v=-\xi^3-\xi_{11}^2+\xi_{12}^2-\xi_{13}^2+\xi_2^2
$$
$$
\Psi^v_3=-\xi^3-\xi_{1}^2+\xi_{21}^2+\xi_{22}^3,\quad \Psi_4^v=-\xi^3-\xi_{1}^2+\xi_{21}^2-\xi_{22}^2+\xi_{23}^2
$$
Finally, we define
$$
N_5^u[u,u,u](t,\xi)=\int_{\xi=\xi_1+\xi_2+\xi_3} e^{it\Phi_2^u}\tilde{u}_1\tilde{u}_2^\ast\tilde{u}_3 d\xi_1d\xi_2,\quad N_5^v[v,v](t,\xi)=\int_{\xi=\xi_1+\xi_2} e^{it\Phi_2^v}\xi \tilde{v}_1\tilde{v}_2 d\xi_1.
$$
As such, we may formally recast equation \eqref{eq:perfilu} as
\begin{align}
	\tilde{u}(t,\xi)&=\tilde{u}(0,\xi)\label{eq:perfilu1} -i\alpha\left[B^u(s,\xi)\right]_{s=0}^{s=t}-i\alpha\sum_{j=0}^4\int_0^t  N_j^u(s,\xi)ds -i\beta\int_0^t  N_5^u(s,\xi)ds
\end{align}
and \eqref{eq:perfilv} as
\begin{align}
	\tilde{v}(t,\xi)&=\tilde{v}(0,\xi) +i\gamma\left[ B^v(s,\xi) \right]_{s=0}^{s=t} +i\gamma \sum_{j=0}^4 \int_0^tN_j^v(s,\xi)ds-i \frac12 \int_0^tN_5^v(s,\xi)ds\label{eq:perfilv1}.
\end{align}
\begin{defi}[Integrated-by-parts strong solution]\label{defi:ibpsol} Given an initial data $(u_0,v_0)\in H^k(\R)\times H^s(\R)$, we say that $(u,v)\in C([0,T]; H^k(\R)\times H^s(\R))$ is an integrated-by-parts strong solution to \eqref{nlskdv} if there exist $\delta^u, \delta^v>0$ such that the profiles $(\tilde{u},\tilde{v})$ satisfy either \eqref{eq:perfilu}-\eqref{eq:perfilv1} or \eqref{eq:perfilu1}-\eqref{eq:perfilv} for all $t\in [0,T]$. 
\end{defi}
Depending on the regularities $(k,s)\in \mathcal{A}$, we will use different definitions. More precisely, for $-2<k-s\le -1$, we will look for solutions of \eqref{eq:perfilu}-\eqref{eq:perfilv1}, and, for $2\le k-s<3$, the suitable definition is \eqref{eq:perfilu1}-\eqref{eq:perfilv}.

The passage from the original system to the integrated-by-parts versions can be made rigorous for smooth solutions and therefore, at high levels of regularity, integrated-by-parts strong solutions and classical strong solutions coincide, {independently of the choice of $\delta^u, \delta^v>0$}. In rougher regularities, however, the two definitions are not \textit{a priori} equivalent.

We are in position of providing a more precise statement of Theorem \ref{thm:well}.

\begin{thm}\label{thm:well2}
	Let $(k,s)\in \mathcal{A}$ and $(u_0,v_0)\in H^k(\R)\times H^s(\R)$. There exist $\delta^u, \delta^v>0$, $T>0$, depending only on $\|u_0\|_{H^k}+\|v_0\|_{H^s}$, and
	$$
	(u,v)\in C([0,T],H^k(\R)) \times C([0,T],H^s(\R)) 
	$$
	such that $(u,v)$ is an integrated-by-parts strong solution to \eqref{nlskdv} over $[0,T]$. Moreover, the solution depends smoothly on the initial data. 
\end{thm}

\begin{nb}
	Together with the ill-posedness result \ref{thm:illposed}, we see that the regularity region $\mathcal{A}$ is sharp in terms of the existence of a $C^2$ flow.
\end{nb}

\begin{nb}
	The solution given by Theorem \ref{thm:well2} can be extended up to a maximal time of existence $T_{max}$. Since the parameters $\delta^u,\delta^v$ and the local time of existence $T$ depend only on the size of the initial data, if $T_{max}<\infty$, then
	$$
	\lim_{t\to T_{max}} \|u(t)\|_{H^k}+\|v(t)\|_{H^s}=+\infty.
	$$
\end{nb}

The continuous dependence on the initial data ensures that the flow $(u_0,v_0)\mapsto (u,v)$ given by Theorem \ref{thm:well2} is the continuous extension of the known flow for smooth initial data. More precisely,

\begin{cor}\label{cor:extension}
	Let $(k,s)\in \mathcal{A}_0$, $M>0$ and
	$$
	B_M^{k,s}=\{(u,v)\in H^k(\R)\times H^s(\R): \|u\|_{H^k}+\|v\|_{H^s}\le M\}.
	$$ 
	Consider the data-to-solution map (\cite{wu})
	$$
	\Theta: B_M^{k,s} \to C([0, T]; B_{2M}^{k,s})
	$$
	$$
	\Theta[u_0,v_0](t)=(u(t),v(t)), \quad t\in [0,T].
	$$
	Then, for any $(k',s')\in \mathcal{A}$ with $k'\le k$ and $s'\le s$, $\Theta$ admits a unique continuous extension to a map
	$$
	\Theta: B_M^{k',s'} \to C([0, T]; B_{2M}^{k',s'}).
	$$
\end{cor}

\begin{nb}
	While Theorem \ref{thm:well2} involves the existence of the parameters $\delta^u$, $\delta^v$, the flow does not depend on the particular choice of parameters. More precisely, if one is able to build a continuous flow of integrated-by-parts solutions for two different pairs $(\delta^u,\delta^v)$ and $(\tilde{\delta}^u,\tilde{\delta}^v)$ (which are \textit{a priori} unrelated), the two flows must coincide.
	In fact, the role of these parameters is merely technical, allowing for a successful construction of the flow.
\end{nb}

\begin{nb}
For regularities $(k,s)\in \mathcal{A}_0$, the notions of integrated-by-parts strong solution and classical strong solution \textit{do} coincide, as the flows constructed in \cite{wu} and Theorem \ref{thm:well2} are the same. As such, integrated-by-parts strong solutions are truly a generalization of the usual concept of solution which allows for rougher levels of regularity.
\end{nb}

\section{Reduction to frequency-restricted estimates (FRE)}\label{sec:fre}

In this section, we reduce the derivation of multilinear estimates in Bourgain spaces to frequency-restricted estimates. To that end, observe that, after recovering the solutions from the profiles
$$\mathcal{F}_xu(t)=e^{-it\xi^2}\tilde{u}(t),\quad \mathcal{F}_xv(t)=e^{it\xi^3}\tilde{v}(t),$$
 the time integrated nonlinear terms in \eqref{eq:perfilu1}-\eqref{eq:perfilv1} can be written in Duhamel form generically as\footnote{For the sake of simplicity, we omit here the possible instances of the conjugation operator $\ast$, with no impairment in the validity of the results below.}
$$\int_0^t e^{it\phi(\xi)}N(s,\xi)ds=\int_0^t e^{i(t-s)\phi(\xi)}e^{is\phi(\xi)}N(s,\xi)ds,$$
where $\phi(\xi)=-\xi^2$ or $\phi(\xi)=\xi^3$ is the phase of the linear Schrödinger or KdV propagator, respectively, depending on which of the two equations we are looking at. As such, writing $$\mathcal{N}(t,\xi)=\mathcal{F}_\xi^{-1}\left(e^{it\phi(\xi)}N(t,\xi)\right),$$
we focus on generic terms of the form
\begin{equation}
	\left(\mathcal{F}_{x}\mathcal{N}[g_1,\dots,g_m]\right)(\xi)= e^{it\phi(\xi)}\int e^{it\Phi}\mathcal{M} \tilde{g}_1\dots \tilde{g}_md\xi_1\dots d\xi_{m-1}
\end{equation}
where $\mathcal{M}$ is a precise Fourier multiplier and 
$$
\Phi=-\phi(\xi)+\sum_{j=1}^m\phi_j(\xi_j).
$$
Therefore we aim to derive general multilinear estimates of the form 
\begin{equation}
	\label{eq:estmulti}
	\| \mathcal{N}[g_1,\dots,g_m]\|_{Z^{s,b'}}\lesssim \prod_{j=1}^{m} \|g_j\|_{Z_j^{s_j,b}},\quad Z, Z_j\in \{X, Y\},\quad s, s_j\in \R
\end{equation}
 for $b=(1/2)^+$ and $b'=(b-1)^+$. First we consider the case of quadratic nonlinearities.

\begin{lem}[Case $m=2$]\label{lem:fre_bi}
	The  estimate \eqref{eq:estmulti} holds if, for some $0<\eta<1$, either:
	\begin{itemize}
		\item for all\footnote{The empty set symbol $\emptyset$ here corresponds to the frequency component $\xi$, without indices.} $j\in\{\emptyset,1,2\}$, $l\neq j$, $\alpha\in \R$ and $1<M\lesssim |\alpha|$,
		\begin{equation}\label{eq:frequad}
			\sup_{\xi_j}\int_\Gamma \frac{|\mathcal{M}|^2\jap{\xi}^{2s}}{\jap{\xi_{1}}^{2s_1}\jap{\xi_{2}}^{2s_2}}\mathbbm{1}_{|\Phi-\alpha|<M} d\xi_{l} \lesssim \jap{\alpha}^\eta M^\eta.
		\end{equation}
		\item for all $\alpha\in \R$ and $M>1$, there exists $j\in\{1,2\}$ such that, with $l\neq j$
		\begin{equation}\label{eq:frequad2}
			\sup_{\xi_j}\int_\Gamma \frac{|\mathcal{M}|^2\jap{\xi}^{2s}}{\jap{\xi_{1}}^{2s_1}\jap{\xi_{2}}^{2s_2}}\mathbbm{1}_{|\Phi-\alpha|<M} d\xi_{l}  \lesssim \jap{\alpha} M^\eta.
		\end{equation}
	\item for all $\alpha\in \R$ and $M>1$,
	\begin{equation}\label{eq:frequad3}
		\sup_{\xi}\int_\Gamma \frac{|\mathcal{M}|^2\jap{\xi}^{2s}}{\jap{\xi_{1}}^{2s_1}\jap{\xi_{2}}^{2s_2}}\mathbbm{1}_{|\Phi-\alpha|<M} d\xi_{1}  \lesssim \jap{\alpha}^\eta M.
	\end{equation}
	\end{itemize}
\end{lem}
%

\begin{proof}
	We start with the bilinear estimate as a consequence of \eqref{eq:frequad}. By duality, it is equivalent to
\begin{align}\label{eq:duality}
	I=\left|\int_{\xi=\xi_1+\xi_2, \tau=\tau_1+\tau_2}\frac{ \mathcal{M} \jap{\xi}^{s}\jap{\tau-\phi(\xi)}^{b'}hh_1h_2}{\jap{\xi_1}^{s_1}\jap{\tau_1-\phi_1(\xi_1)}^b\jap{\xi_2}^{s_2}\jap{\tau_2-\phi_2(\xi_2)}^b}d\xi_1d\tau_1d\xi  d\tau \right|
	\lesssim \|h\|_{L^2}\|h_1\|_{L^2}\|h_2\|_{L^2}.
\end{align}
We decompose the domain into three regions, depending on whether
\begin{equation}\label{eq:decompdom}
	|\tau-\phi(\xi)|,\quad |\tau_1-\phi_1(\xi_1)|\mbox{ or }|\tau_2-\phi_2(\xi_2)|
\end{equation}
is the largest of the three. Assume that $|\tau-\phi(\xi)|$ is the largest. Since
$$
\Phi=\left(\tau-\phi(\xi)\right) - \left(\tau_1-\phi_1(\xi_1)\right)-\left(\tau_2-\phi_2(\xi_2)\right),
$$ 
this implies that $|\Phi|\lesssim |\tau-\phi(\xi)|$. By Cauchy-Schwarz,
\begin{align}\label{eq:cs}
	I&\lesssim \left(\int\frac{|\mathcal{M}|^2\jap{\xi}^{2s}\jap{\tau-\phi(\xi)}^{2b'}|h|^2d\xi_1d\tau_1d\xi d\tau}{\jap{\xi_1}^{2s_1}\jap{\tau_1-\phi_1(\xi_1)}^{2b}\jap{\xi_2}^{2s_2}\jap{\tau_2-\phi_2(\xi_2)}^{2b}} \right)^{\frac{1}{2}}\|h_1\|_{L^2}\|h_2\|_{L^2}\\&\lesssim\sup_{\tau,\xi} \left(\int\frac{|\mathcal{M}|^2\jap{\xi}^{2s}\jap{\tau-\phi(\xi)}^{2b'}d\xi_1d\tau_1}{\jap{\xi_1}^{2s_1}\jap{\tau_1-\phi_1(\xi_1)}^{2b}\jap{\xi_2}^{2s_2}\jap{\tau_2-\phi_2(\xi_2)}^{2b}}\right)^{\frac{1}{2}}
	\|h\|_{L^2}\|h_1\|_{L^2}\|h_2\|_{L^2}.
\end{align}
Since $2b>1$, integration in $\tau_1$ yields 
\begin{equation}\label{eq:integratetau}
	\int_{\tau=\tau_1+\tau_2} \frac{1}{\jap{\tau_1-\phi_1(\xi_1)}^{2b}\jap{\tau_2-\phi_2(\xi_2)}^{2b}}d\tau_1 \lesssim \frac{1}{\jap{\tau-\phi_1(\xi_1)-\phi_2(\xi_2)}^{2b}}= \frac{1}{\jap{\tau-\phi(\xi) - \Phi}^{2b}}.
\end{equation}
Therefore
\begin{align*}
	I&\lesssim \sup_{\tau,\xi}\left(\int_{\xi=\xi_1+\xi_2}\frac{|\mathcal{M}|^2\jap{\xi}^{2s}}{\jap{\xi_1}^{2s_1}\jap{\xi_2}^{2s_2}}\frac{\jap{\tau-\phi(\xi)}^{2b'}}{\jap{\tau-\phi(\xi) - \Phi}^{2b}}d\xi_1\right)^{\frac{1}{2}}\|h\|_{L^2}\|h_1\|_{L^2}\|h_2\|_{L^2}\\&\lesssim \sup_{\alpha,\xi}\left(\int_{|\Phi|\lesssim |\alpha|, \xi=\xi_1+\xi_2}\frac{|\mathcal{M}|^2\jap{\xi}^{2s}}{\jap{\xi_1}^{2s_1}\jap{\xi_2}^{2s_2}}\frac{\jap{\alpha}^{2b'}}{\jap{ \Phi-\alpha}^{2b}}d\xi_1\right)^{\frac{1}{2}}\|h\|_{L^2}\|h_1\|_{L^2}\|h_2\|_{L^2}.
\end{align*}
Performing a dyadic decomposition in $\Phi-\alpha$ and using the frequency-restricted estimate \eqref{eq:frequad},
\begin{align*}
	I&\lesssim \sup_{\alpha,\xi}\sum_{\stackrel{M\ \text{dyadic}}{ M\lesssim |\alpha|} }\left(\int\frac{|\mathcal{M}|^2\jap{\xi}^{2s}}{\jap{\xi_1}^{2s_1}\jap{\xi_2}^{2s_2}}\frac{\jap{\alpha}^{2b'}}{M^{2b}}\mathbbm{1}_{|\Phi-\alpha|<M}d\xi_1\right)^{\frac{1}{2}}\|h\|_{L^2}\|h_1\|_{L^2}\|h_2\|_{L^2} \\&\lesssim \sup_\alpha\sum_{\stackrel{M\ \text{dyadic}}{ M\lesssim |\alpha|} } \jap{\alpha}^{b'+\eta/2}M^{-b+\eta/2}\|h\|_{L^2}\|h_1\|_{L^2}\|h_2\|_{L^2}\lesssim \|h\|_{L^2}\|h_1\|_{L^2}\|h_2\|_{L^2}
\end{align*}
and the claim follows. In the case where $|\tau_1-\phi_1(\xi_1)|$ is the largest in \eqref{eq:decompdom}, we have $|\Phi|\lesssim |\tau_1-\phi_1(\xi_1)|$. Applying Cauchy-Schwarz and proceeding as above,
\begin{align}
	I&\lesssim \left(\int\frac{|\mathcal{M}|^2\jap{\xi}^{2s}\jap{\tau-\phi(\xi)}^{2b'}|h_1|^2d\xi_1d\tau_1d\xi d\tau}{\jap{\xi_1}^{2s_1}\jap{\tau_1-\phi_1(\xi_1)}^{2b}\jap{\xi_2}^{2s_2}\jap{\tau_2-\phi_2(\xi_2)}^{2b}} \right)^{\frac{1}{2}}\|h\|_{L^2}\|h_2\|_{L^2}\\&\lesssim\sup_{\tau_1,\xi_1} \left(\int\frac{|\mathcal{M}|^2\jap{\xi}^{2s}\jap{\tau-\phi(\xi)}^{2b'}d\xi d\tau}{\jap{\xi_1}^{2s_1}\jap{\tau_1-\phi_1(\xi_1)}^{2b}\jap{\xi_2}^{2s_2}\jap{\tau_2-\phi_2(\xi_2)}^{2b}}\right)^{\frac{1}{2}}
	\|h\|_{L^2}\|h_1\|_{L^2}\|h_2\|_{L^2}\\&\lesssim\label{eq:estlemma31} \sup_{\tau_1,\xi_1}\left(\int_{\xi=\xi_1+\xi_2}\frac{|\mathcal{M}|^2\jap{\xi}^{2s}}{\jap{\xi_1}^{2s_1}\jap{\xi_2}^{2s_2}}\frac{\jap{\tau_1-\phi_1(\xi_1)}^{-2b}}{\jap{\tau_1-\phi_1(\xi_1) + \Phi}^{-2b'}}d\xi\right)^{\frac{1}{2}}\|h\|_{L^2}\|h_1\|_{L^2}\|h_2\|_{L^2}\\&\lesssim \sup_{\alpha,\xi_1}\left(\int_{|\Phi|\lesssim |\alpha|, \xi=\xi_1+\xi_2}\frac{|\mathcal{M}|^2\jap{\xi}^{2s}}{\jap{\xi_1}^{2s_1}\jap{\xi_2}^{2s_2}}\frac{\jap{\alpha}^{-2b}}{\jap{ \Phi-\alpha}^{-2b'}}d\xi\right)^{\frac{1}{2}}\|h\|_{L^2}\|h_1\|_{L^2}\|h_2\|_{L^2}\\&\lesssim \sup_{\alpha,\xi_1}\sum_{\stackrel{M\ \text{dyadic}}{ M\lesssim |\alpha|} }\left(\int\frac{|\mathcal{M}|^2\jap{\xi}^{2s}}{\jap{\xi_1}^{2s_1}\jap{\xi_2}^{2s_2}}\frac{\jap{\alpha}^{-2b}}{M^{-2b'}}\mathbbm{1}_{|\Phi-\alpha|<M}d\xi\right)^{\frac{1}{2}}\|h\|_{L^2}\|h_1\|_{L^2}\|h_2\|_{L^2} \\&\lesssim \sup_\alpha \sum_{\stackrel{M\ \text{dyadic}}{ M\lesssim |\alpha|} } \jap{\alpha}^{-b+\eta/2}M^{b'+\eta/2}\|h\|_{L^2}\|h_1\|_{L^2}\|h_2\|_{L^2}\lesssim \|h\|_{L^2}\|h_1\|_{L^2}\|h_2\|_{L^2}
\end{align}
The case where $|\tau_2-\phi_2(\xi_2)|$ is the largest in \eqref{eq:decompdom} is treated analogously. Thus \eqref{eq:frequad} implies the bilinear estimate \eqref{eq:estmulti}.

Concerning the bilinear estimate through \eqref{eq:frequad2} or \eqref{eq:frequad3}, one follows the steps above without decomposing the domain as in \eqref{eq:decompdom} and consequently without the extra condition $|\Phi|\lesssim |\alpha|$. More precisely, by \eqref{eq:cs} and \eqref{eq:integratetau},
\begin{align*}
	I&\lesssim \sup_{\tau,\xi}\left(\int_{\xi=\xi_1+\xi_2}\frac{|\mathcal{M}|^2\jap{\xi}^{2s}}{\jap{\xi_1}^{2s_1}\jap{\xi_2}^{2s_2}}\frac{\jap{\tau-\phi(\xi)}^{2b'}}{\jap{\tau-\phi(\xi) + \Phi}^{2b}}d\xi_1\right)^{\frac{1}{2}}\|h\|_{L^2}\|h_1\|_{L^2}\|h_2\|_{L^2}\\&\lesssim \sup_{\alpha,\xi}\left(\int_{ \xi=\xi_1+\xi_2}\frac{|\mathcal{M}|^2\jap{\xi}^{2s}}{\jap{\xi_1}^{2s_1}\jap{\xi_2}^{2s_2}}\frac{\jap{\alpha}^{2b'}}{\jap{ \Phi-\alpha}^{2b}}d\xi_1\right)^{\frac{1}{2}}\|h\|_{L^2}\|h_1\|_{L^2}\|h_2\|_{L^2}.
\end{align*}
Decomposing dyadically in $\Phi-\alpha$ and applying \eqref{eq:frequad3},
\begin{align*}
	I&\lesssim \sup_{\alpha,\xi}\sum_{M\ \text{dyadic}} \left(\int\frac{|\mathcal{M}|^2\jap{\xi}^{2s}}{\jap{\xi_1}^{2s_1}\jap{\xi_2}^{2s_2}}\frac{\jap{\alpha}^{2b'}}{M^{2b}}\mathbbm{1}_{|\Phi-\alpha|<M}d\xi_1\right)^{\frac{1}{2}}\|h\|_{L^2}\|h_1\|_{L^2}\|h_2\|_{L^2} \\&\lesssim \sup_\alpha\sum_{M\ \text{dyadic}} \jap{\alpha}^{b'+\eta/2}M^{-b+1/2}\|h\|_{L^2}\|h_1\|_{L^2}\|h_2\|_{L^2}\lesssim \|h\|_{L^2}\|h_1\|_{L^2}\|h_2\|_{L^2}
\end{align*}
and thus \eqref{eq:frequad3} also implies \eqref{eq:estmulti}. On the other hand, proceeding as in \eqref{eq:estlemma31} and using \eqref{eq:frequad2},
\begin{align}
	I&\lesssim  \sup_{\alpha,\xi_1}\sum_{M\ \text{dyadic} }\left(\int\frac{|\mathcal{M}|^2\jap{\xi}^{2s}}{\jap{\xi_1}^{2s_1}\jap{\xi_2}^{2s_2}}\frac{\jap{\alpha}^{-2b}}{M^{-2b'}}\mathbbm{1}_{|\Phi-\alpha|<M}d\xi\right)^{\frac{1}{2}}\|h\|_{L^2}\|h_1\|_{L^2}\|h_2\|_{L^2} \\&\lesssim \sup_\alpha \sum_{M\ \text{dyadic} } \jap{\alpha}^{-b+1/2}M^{b'+\eta/2}\|h\|_{L^2}\|h_1\|_{L^2}\|h_2\|_{L^2}\lesssim \|h\|_{L^2}\|h_1\|_{L^2}\|h_2\|_{L^2}.
\end{align}
Therefore \eqref{eq:frequad2} implies \eqref{eq:estmulti} and the proof is finished.
\end{proof}

For higher-order nonlinear terms, one could apply a Cauchy-Schwarz argument. However, this is usually not optimal in terms of regularity. In \cite{COS}, using a standard interpolation argument, the first and third authors derived a Schur's test-type criterion for the validity of \eqref{eq:estmulti}:

\begin{lem}[Case $m\ge 3$, \cite{COS}]\label{lem:fre_tri}
	The estimate \eqref{eq:estmulti} holds if, for some $0<\eta<1$, there exist a proper nonempty subset $A\subset \{\emptyset,1,\dots, m\}$ and multipliers $\mathcal{M}_1,\mathcal{M}_2\ge 0$ such that
		$$
		\sup_{\xi_{j\in A}} \int_{\Gamma} \mathcal{M}_1 \mathbbm{1}_{|\Psi-\alpha|<M} d\xi_{j\notin A} + \sup_{\xi_ j\notin A} \int_{\Gamma} \mathcal{M}_2 \mathbbm{1}_{|\Psi-\alpha|<M} d\xi_{j\in A} \lesssim M^\eta
		$$
		and
		$$
		\mathcal{M}_1\mathcal{M}_2=\frac{|\mathcal{M}|^2\jap{\xi}^{2s}}{\prod_{j=1}^k\jap{\xi_{j}}^{2s_j}}.
		$$
\end{lem}

In both Lemma \ref{lem:fre_bi} and Lemma \ref{lem:fre_tri}, the frequency-restricted estimate requires $\eta<1$. The next lemma shows how one can reduce the problem to $\eta=1$ at the expense of slightly increasing the multiplier.

\begin{lem}\label{lem:eta1}
	Let $K:\R^n\times \R^m\times \R\to \R^+$ be a positive measurable function such that, for some $N\in \mathbb{N}$,
	$$
	|K(x,y,M)|\lesssim \max\{1,|x|,|y|\}^N,\quad \forall x\in\R^n,\ y\in \R^m, \ M\in \R.
	$$
	Suppose that
	$$
	\sup_y \int K(x,y,M)\max\{1,|x|,|y|\}^{0^+} dx \lesssim M,\quad \mbox{for all }M>1.
	$$
	Then there exists $0<\eta<1$ such that
	$$
	\sup_y \int K(x,y,M) dx \lesssim M^{\eta},\quad \mbox{for all } M>1.
	$$
\end{lem}
\begin{proof}
	The proof is an immediate consequence of Hölder's inequality. For $p>1$,
	\begin{align*}
	\sup_y \int K(x,y,M) dx &\lesssim \sup_y \int K(x,y,M)^{\frac{1}{p}}K(x,y,M)^{\frac{1}{p'}} dx \\&\lesssim \sup_y \int \frac{K(x,y,M)^{\frac{1}{p}}\max\{1,|x|,|y|\}^{\frac{N+n+1}{p'}}}{\max\{1,|x|,|y|\}^{\frac{n+1}{p'}}} dx \\&\lesssim \sup_y \left(\int K(x,y,M)\max\{1,|x|,|y|\}^{(N+n+1)(p-1)} dx\right)^{\frac{1}{p}}\left(\int\frac{dx}{\jap{x}^{n+1}}\right)^{\frac{p-1}{p}}\lesssim M^{1/p},
	\end{align*}
as long as we take $p$ sufficiently close to 1.
\end{proof}

From now on, whenever one considers a multilinear estimate in the variables $\xi_1,\dots, \xi_m$, we set
\begin{equation}\label{eq:XI}
	\Xi=\max_{1\le j\le m} \jap{\xi_j}.
\end{equation}
The combination of the previous lemma with Lemma \ref{lem:fre_bi} yields
\begin{lem}[Case $m=2$]\label{lem:fre_bi1}
	The  estimate \eqref{eq:estmulti} holds if
\begin{equation}\label{eq:condm}
		|\mathcal{M}|\lesssim \Xi^N\quad \mbox{for some }N\in\mathbb{N}
\end{equation}
	and either:
	\begin{itemize}
		\item for all $j\in\{\emptyset,1,2\}$, $l\neq j$, $\alpha\in \R$ and $1<M\lesssim |\alpha|$,
		\begin{equation}\label{eq:frequad1}
			\sup_{\xi_j}\int_\Gamma \frac{|\mathcal{M}|^2\jap{\xi}^{2s}}{\jap{\xi_{1}}^{2s_1}\jap{\xi_{2}}^{2s_2}}\Xi^{0^+}\mathbbm{1}_{|\Phi-\alpha|<M} d\xi_{l} \lesssim \jap{\alpha}M.
		\end{equation}
		\item for all $\alpha\in \R$ and $M>1$, there exists $j\in\{1,2\}$ such that, with $l\neq j$
		\begin{equation}\label{eq:frequad21}
			\sup_{\xi_j}\int_\Gamma \frac{|\mathcal{M}|^2\jap{\xi}^{2s}}{\jap{\xi_{1}}^{2s_1}\jap{\xi_{2}}^{2s_2}}\Xi^{0^+}\mathbbm{1}_{|\Phi-\alpha|<M} d\xi_{l}  \lesssim \jap{\alpha} M.
		\end{equation}
		\item for all $\alpha\in \R$ and $M>1$,
		\begin{equation}\label{eq:frequad31}
			\sup_{\xi}\int_\Gamma \frac{|\mathcal{M}|^2\jap{\xi}^{2s}}{\jap{\xi_{1}}^{2s_1}\jap{\xi_{2}}^{2s_2}}\Xi^{0^+}\mathbbm{1}_{|\Phi-\alpha|<M} d\xi_{1}  \lesssim \jap{\alpha} M.
		\end{equation}
	\end{itemize}
\end{lem}
Analogously, Lemmas \ref{lem:fre_tri} and \ref{lem:eta1} give the following useful lemma.

\begin{lem}[Case $m\ge 3$]\label{lem:fre_tri1}
	The estimate \eqref{eq:estmulti} holds if 
\begin{equation}\label{eq:condm1}
	|\mathcal{M}|\lesssim \Xi^N\quad \mbox{for some }N\in\mathbb{N},
\end{equation}
	 there exist a proper nonempty subset $A\subset \{\emptyset,1,\dots, m\}$ and multipliers $\mathcal{M}_1,\mathcal{M}_2\ge 0$ such that
	$$
	\sup_{\xi_{j\in A}} \int_{\Gamma} \mathcal{M}_1\Xi^{0^+} \mathbbm{1}_{|\Psi-\alpha|<M} d\xi_{j\notin A} + \sup_{\xi_ j\notin A} \int_{\Gamma} \mathcal{M}_2\Xi^{0^+} \mathbbm{1}_{|\Psi-\alpha|<M} d\xi_{j\in A} \lesssim M
	$$
	and
	$$
	\mathcal{M}_1\mathcal{M}_2=\frac{|\mathcal{M}|^2\jap{\xi}^{2s}}{\prod_{j=1}^k\jap{\xi_{j}}^{2s_j}}.
	$$
\end{lem}

\section{Multilinear estimates}\label{sec:multi}

In this section, we prove the multilinear estimates associated to each of the nonlinear terms in \eqref{eq:perfilu1}-\eqref{eq:perfilv1}. 
In the computations performed below, we will often apply Lemmas \ref{lem:fre_bi1} and \ref{lem:fre_tri1}. In either case, it will be trivial to check that the multiplier $\mathcal{M}$ satisfies \eqref{eq:condm} (or \eqref{eq:condm1}) and we focus on the derivation of the frequency-restricted estimates. 

Generally speaking, the procedure will be as follows. First, we perform a change of variables in order to integrate in the resonance function $\Phi$ itself. This will introduce some frequency decay coming from the jacobian of the change of variables. In a second step, one either bounds the resulting multiplier by 1 or by $\Phi$ itself. In the first case, the integration in $\Phi$ gives the desired power of $M$ and the estimate is done. In the second case, one must apply \eqref{eq:frequad1} in order to have access to the bound $|\Phi|\lesssim |\alpha|$. Using this bound and then integrating in $\Phi$, we find $|\alpha|M$, which is the required estimate. 

 We write the nonlinear terms in physical space as
$$
\mathcal{N}_j^u = \mathcal{F}_\xi^{-1}\left(e^{-it\xi^2}N_j^u\right),\quad \mathcal{N}_j^v = \mathcal{F}_\xi^{-1}\left(e^{it\xi^3}N_j^v\right),\quad j=0,\dots, 4.
$$
Throughout this section, $b=(1/2)^+$, $b'=(b-1)^+$ and $\Xi$ is defined as in \eqref{eq:XI}.

\begin{lem}\label{lem:1}
For $s>-1$ and $\epsilon<\min\{(s+1)/2,1/2\}$, 
\begin{equation}
	\left\| \mathcal{N}_0^u[u,v]\right\|_{X^{k+\epsilon,b'}}\lesssim \|u\|_{X^{k,b}}\|v\|_{Y^{s,b}}.
\end{equation}	
\end{lem}
\begin{proof}
	We recall the resonance function $\Phi=\Phi_1^u=\xi^2-\xi_1^2+\xi_2^3=\xi_2(\xi_1+\xi+\xi_2^2)$ and define the multiplier
	$$
	\mathcal{M}=\frac{\jap{\xi}^{k+\epsilon}}{\jap{\xi_1}^k\jap{\xi_2}^s}.
	$$ 
	If all frequencies are smaller than $1/\delta^u$, the FRE is direct. As we are integrating over $U_\xi^c$, there are only two cases:
	\vskip10pt
	\noindent\textbf{\underline{Case A.}} $|\xi_1|\sim |\xi_2|\sim |\xi|$. We use \eqref{eq:frequad3}: since
	$$
	\frac{\partial \Phi}{\partial \xi_1} = -2\xi_1-3\xi_2^2 \sim \xi_2^2,
	$$
	the change of variables $\xi_1\mapsto \Phi$ gives
	$$
	\sup_\xi \int \mathcal{M}^2\fia d\xi_1 \lesssim \sup_\xi \int \xi^{2\epsilon-2s-2}\fia d\Phi \lesssim M.
	$$
	\vskip10pt
	\noindent\textbf{\underline{Case B.}} $|\xi_2|\ll |\xi_1|\simeq |\xi|$, in which case
	$$
	\mathcal{M}^2\sim \jap{\xi_2}^{-2s}|\xi|^{2\epsilon}.
	$$
	If $|\xi_2|<1$, we use \eqref{eq:frequad3}:
	$$
	\sup_\xi \int \mathcal{M}^2\fia d\xi_1 \lesssim \sup_\xi \int \jap{\xi_2}^{-2s}|\xi|^{2\epsilon}\fia \frac{d\Phi}{|\xi_1|} \lesssim M.
	$$
	If $|\xi_2|>1$, we divide into subcases:
	\begin{enumerate}
		\item $|\xi_1+\xi+\xi_2^2|\gtrsim |\xi|$ and $|2\xi_1 + 3\xi_2^2|\gtrsim |\xi_1|$, which implies
		$$
		|\Phi|\gtrsim |\xi_2\xi|.
		$$
		We apply \eqref{eq:frequad1}. First fix $\xi$. If $s<0$, then
		$$
		|\xi_2|^{-2s}|\xi_1|^{2\epsilon-1+0^+}\lesssim |\xi_2|^{-2s}|\xi|^{1+2s}\lesssim |\Phi|
		$$
		while, if $s>0$,
		$$
		|\xi_2|^{-2s}|\xi_1|^{2\epsilon-1+0^+}\lesssim |\xi_2|^{-2s}|\xi|\lesssim |\Phi|.
		$$
		In either case,
		\begin{align*}
			\sup_\xi \int \mathcal{M}^{2}\Xi^{0^+} \fia d\xi_1 &\lesssim \sup_\xi \int |\xi_2|^{-2s}|\xi|^{-1+2\epsilon+0^+} \fia d\Phi \\&\lesssim \sup_\xi \int |\Phi| \fia d\Phi \lesssim |\alpha|M.
		\end{align*}
	By symmetry, the same goes when fixing $\xi_1$. Finally, fixing $\xi_2$, since $$|\partial_{\xi_1}\Phi|\sim |\xi_2|\quad \mbox{ and }\quad
	|\xi_2|^{-2s-1}|\xi|^{2\epsilon}\lesssim |\Phi|^{1^-},
	$$
		\begin{align*}
			\sup_{\xi_2} \int \mathcal{M}^{2}\Xi^{0^+}\fia d\xi_1 & \lesssim \sup_{\xi_2} \int |\xi_2|^{-2s-1}|\xi|^{2\epsilon+0^+}\fia d\Phi\\&\lesssim \sup_{\xi_2} \int |\Phi| \fia d\Phi \lesssim |\alpha|M.
		\end{align*}
		\item $|\xi+\xi_1+\xi_2^2|\ll |\xi|$ or $|2\xi_1 + 3\xi_2^2|\ll |\xi_1|$, which implies $|\xi_2|^2 \sim |\xi|$. In particular,
		$$
		|\xi|^{3/2}\lesssim |\Phi| \lesssim \jap{\alpha}M.
		$$
		Applying \eqref{eq:frequad3},
		\begin{align*}
			\sup_\xi \int \mathcal{M}^2\fia d\xi_2  &\lesssim \sup_\xi \int |\xi_2|^{-2s}|\xi|^{2\epsilon} \fia d\xi_2 \lesssim \sup_\xi |\xi|^{-s+1/2+2\epsilon}\\ & \lesssim(\jap{\alpha}M)^{(1+4\epsilon-2s)/3}.
		\end{align*}
		
	\end{enumerate}
\end{proof}

Over the problematic region $U_\xi$, one can easily identify where the restriction $k-s<2$ arises.
\begin{lem}\label{lem:probU}
	For $k-s<2$ and $\epsilon<2-k+s$,
	\begin{equation}
		\left\|uv- \mathcal{N}_0^u[u,v]\right\|_{X^{k+\epsilon,b'}}\lesssim (\delta^u)^{0^+}\|u\|_{X^{k,b}}\|v\|_{Y^{s,b}}.
	\end{equation}
\end{lem}
\begin{proof}
	In frequency space, the integration is made over $U_\xi$. Over this region, $|\xi_1|\ll |\xi_2|\simeq |\xi|$ and $1/\delta^u<|\xi|$, and thus
	$$
	\mathcal{M}^2\lesssim (\delta^u)^{0^+}|\xi|^{2k-2s+2\epsilon},\quad |\Phi|\sim |\xi_2|^3.
	$$
	We apply \eqref{eq:frequad1}. Fixing $\xi$,
	\begin{align*}
		\sup_\xi \int \mathcal{M}^{2}\Xi^{0^+} \fia d\xi_1 &\lesssim (\delta^u)^{0^+} \sup_\xi \int |\xi|^{2k-2s+2\epsilon-2+0^+}\fia d\Phi \\&\lesssim (\delta^u)^{0^+}\int |\Phi|^{(2k-2s+2\epsilon-2+0^+)/3}\fia d\Phi \\&\lesssim (\delta^u)^{0^+} |\alpha|M\quad \mbox{if}\quad k-s+\epsilon<5/2.
	\end{align*}
	The same goes when one fixes $\xi_1$. Finally, fixing $\xi_2$,
	\begin{align*}
		\sup_{\xi_2}\int \mathcal{M}^{2}\Xi^{0^+} \fia d\xi &\lesssim (\delta^u)^{0^+}\sup_{\xi_2}\int |\xi|^{2k-2s+2\epsilon-1+0^+}\fia d\Phi \\&\lesssim (\delta^u)^{0^+}\int |\Phi|^{(2k-2s+2\epsilon-1+0^+)/3}\fia d\Phi \\&\lesssim (\delta^u)^{0^+} |\alpha|M\quad \mbox{if}\quad k-s+\epsilon<2.
	\end{align*} 
\end{proof}

\begin{lem}\label{lem:2}
	For $s<\min\{4k,2k+3/2\}$ and $\epsilon<\min\{4k-s, 2k-s+3/2\}$, 
		\begin{equation}
		\left\| \mathcal{N}_0^v[u_1,u_2]\right\|_{Y^{s+\epsilon,b'}}\lesssim \|u_1\|_{X^{k,b}}\|u_2\|_{X^{k,b}}.
	\end{equation}	
\end{lem}
\begin{proof}
	The resonance function is $\Phi=\Phi_1^v=-\xi^3-\xi_1^2+\xi_2^2=-\xi(\xi^2-\xi_2+\xi_1)$. W.l.o.g., suppose that $|\xi_1|\ge |\xi_2|$. The case where $|\xi_1|<1/\delta^v$ is trivial, we focus on $|\xi_1|>1/\delta^v$.
Recall that we are restricted to the region $|\xi|\ll |\xi_1|\simeq |\xi_2|$.

\vskip10pt
\noindent \textbf{\underline{Case A.}} $|\xi_1|\simeq |\xi_2|\gtrsim |\xi|^2$. We consider two possibilities:
	\begin{enumerate}
		\item $|3\xi^2-2\xi_2|\gtrsim |\xi_2|$. Then we apply \eqref{eq:frequad21}:
		$$
		\sup_{\xi_1} \int \frac{\jap{\xi}^{2s+2\epsilon}|\xi|^2}{\jap{\xi_1}^{4k}}\Xi^{0^+}\fia d\xi \lesssim \sup_{\xi_1}\int |\xi_1|^{s+\epsilon-4k+0^+}\fia d\Phi\lesssim M.
		$$
		\item  $|3\xi^2-2\xi_2|\ll |\xi_2|$, which implies $|3\xi^2+2\xi_1|\gtrsim |\xi_1|$. Applying \eqref{eq:frequad21},
		$$
		\sup_{\xi_2} \int \frac{\jap{\xi}^{2s+2\epsilon}|\xi|^2}{\jap{\xi_2}^{4k}}\Xi^{0^+}\fia d\xi\lesssim \sup_{\xi_2} \int |\xi_2|^{s+\epsilon-4k+0^+}\fia d\Phi \lesssim M.
		$$
	\end{enumerate}
\noindent \textbf{\underline{Case B.}} $|\xi_1|\simeq |\xi_2|\ll |\xi|^2$, which implies $|\Phi|\sim |\xi|^3$. We apply \eqref{eq:frequad1}. First, fixing $\xi_1$, $|\partial_\xi \Phi|\gtrsim |\xi|^2$ and thus
$$
\sup_{\xi_1} \int \frac{\jap{\xi}^{2s+2\epsilon}|\xi|^2}{\jap{\xi_1}^{4k}}\Xi^{0^+}\fia d\xi \lesssim \sup_{\xi_1} \int |\xi|^{2s+2\epsilon-4k+0^+}\fia d\Phi \lesssim |\alpha|^{(2s+2\epsilon-4k+0^+)/3}M.
$$
The same computation holds when $\xi_2$ is fixed. Finally, fixing $\xi$, since $|\partial_{\xi_1}\Phi|\sim |\xi\xi_1|\gtrsim |\xi|^2$, 
$$
\sup_{\xi} \int \frac{\jap{\xi}^{2s+2\epsilon}|\xi|^2}{\jap{\xi_1}^{4k}}\Xi^{0^+}\fia d\xi_1 \lesssim \sup_{\xi} \int |\xi|^{2s+2\epsilon-4k+0^+}\fia d\Phi \lesssim |\alpha|^{(2s+2\epsilon-4k+0^+)/3}M.
$$
\end{proof}
\begin{lem}\label{lem:probV}
	For $k-s>-1$ and $\epsilon<k-s+1$, 
	\begin{equation}
		\left\|\partial_x(u_1\overline{u_2})- \mathcal{N}_0^v[u_1,u_2]\right\|_{Y^{s+\epsilon,b'}}\lesssim (\delta^v)^{0^+}\|u_1\|_{X^{k,b}}\|u_2\|_{X^{k,b}}.
	\end{equation}	
\end{lem}
\begin{proof}
	First, observe that, in frequency space, the integration is being held over $V_\xi$, where $|\xi|\gtrsim |\xi_2|$ and $|\xi|>1/\delta^v$. In particular, $|\Phi|\sim |\xi|^3$. If we apply \eqref{eq:frequad1}:
	\begin{align*}
		\sup_{\xi_2} \int \frac{|\xi|^{2s+2\epsilon+2}}{\jap{\xi_1}^{2k}\jap{\xi_2}^{2k}}\Xi^{0^+}\fia d\xi &\lesssim \sup_{\xi_2} \int |\xi|^{2s+2\epsilon-2k+0^+}\fia d\Phi \\&\lesssim (\delta^v)^{0^+}\int |\Phi|^{(2s+2\epsilon-2k+0^+)/3}\fia d\Phi \\&\lesssim (\delta^v)^{0^+}\jap{\alpha}M\quad \mbox{if}\quad {s+\epsilon-k<3/2};
	\end{align*}
	the estimate for $\xi_1$ fixed is analogous; and
	\begin{align*}
		\sup_{\xi} \int \frac{|\xi|^{2s+2\epsilon+2}}{\jap{\xi_1}^{2k}\jap{\xi_2}^{2k}}\Xi^{0^+}\fia d\xi_1 &\lesssim \sup_{\xi_2} \int |\xi|^{2s+2\epsilon-2k+1+0^+}\fia d\Phi \\&\lesssim (\delta^v)^{0^+}\int |\Phi|^{(2s+2\epsilon-2k+1+0^+)/3}\fia d\Phi \\&\lesssim (\delta^v)^{0^+}\jap{\alpha}M\quad \mbox{if}\quad {s+\epsilon-k<1}.
	\end{align*}
\end{proof}

\begin{nb}
	The  four lemmas above give an alternative proof of Lemma \ref{lem:est_wu} through a complete analysis in frequency variables, with a quantified nonlinear smoothing effect.
\end{nb}

Having understood, in Lemmas \ref{lem:probU} and $\ref{lem:probV}$, the precise obstructions, it is now an easy task to build counterexamples for estimates \eqref{eq:estwu1} and \eqref{eq:estwu2}:

\begin{cor}\label{cor:counter}
For $k-s>2$, $b=(1/2)^+$ and $b'=(-1/2)^+$, estimate \eqref{eq:estwu1} fails.
\end{cor}
\begin{proof}
	It suffices to prove that the dualized version of \eqref{eq:estwu1},
\begin{equation}\label{eq:counterdual}
		\left|\int \frac{\jap{\xi}^k\jap{\tau-\xi^2}^{b'}h_1h_2h}{\jap{\xi_1}^k\jap{\tau_1-\xi_1^2}^{b}\jap{\xi_2}^s\jap{\tau_2+\xi_2^3}^{b}}d\tau d\xi d\tau_1 d\xi_1\right| \lesssim \|h_1\|_{L^2}\|h_2\|_{L^2} \|h\|_{L^2},
\end{equation}
	fails. Take
\begin{equation}\label{eq:counter1}	h_1(\tau_1,\xi_1)=\mathbbm{1}_{[0,1]}(\xi_1)\mathbbm{1}_{[0,1]}(\tau_1),\quad 	h_2(\tau_2,\xi_2)=\mathbbm{1}_{[N,N+1]}(\xi_2)\mathbbm{1}_{[N^2, N^2+1]}(\tau_2),
\end{equation}
	\begin{equation}\label{eq:counter2}
		h(\tau,\xi)=\mathbbm{1}_{[N,N+1]}(\xi)\mathbbm{1}_{[N^2, N^2+1]}(\tau).
	\end{equation}
Then $\|h_1\|_{L^2}=\|h_2\|_{L^2}= \|h\|_{L^2}=1$, 
$$\tau_1-\xi_1^2 \sim 1, \quad \tau_2+\xi_2^3 \sim N^3, \quad \tau-\xi^2 \sim N
$$
and plugging into \eqref{eq:counterdual} yields
$$
N^{b'-3b+k-s}\lesssim 1.
$$
As $b=(1/2)^+$ and $b'=(-1/2)^+$, the contradiction follows for large $N$.
\end{proof}
\begin{cor}\label{cor:counter2}
	For $k-s<-1$, $b=(1/2)^+$ and $b'=(-1/2)^+$, estimate \eqref{eq:estwu2} fails.
\end{cor}
\begin{proof}
	As in the previous proof, we show that
\begin{equation}\label{eq:counterdual2}
		\left|\int \frac{\xi\jap{\xi}^s\jap{\tau+\xi^3}^{b'}h_1h_2h}{\jap{\xi_1}^k\jap{\tau_1-\xi_1^2}^{b}\jap{\xi_2}^k\jap{\tau_2-\xi_2^2}^{b}}d\tau d\xi d\tau_1 d\xi_1\right| \lesssim \|h_1\|_{L^2}\|h_2\|_{L^2} \|h\|_{L^2}
\end{equation}
	cannot hold. Choose $h_1,h_2, h$ as in \eqref{eq:counter1} and \eqref{eq:counter2}. Then \eqref{eq:counterdual2} would yield
	$$
	N^{1+s-k+3b'-b}\lesssim 1,
	$$
	which does not hold for large $N$.
\end{proof}

In the next lemmas, we prove the necessary multilinear estimates associated with the $N_j^u$ terms, $j=1,\dots,4$. 

The proof of each estimate will follow from the application of Lemma \ref{lem:fre_tri1}, with suitable choices of multipliers $\mathcal{M}_1, \mathcal{M}_2$ and of the set $\mathcal{A}$. This choice is made taking into account two factors: first, we are restricted to the region $|\xi_1|\ll |\xi_2|\simeq |\xi|$; second, we aim to perform a change of variables in the frequency-restricted estimates in order to integrate in the resonance function itself. By considering appropriate pairings of variables, we can ensure that this change of variables can be made, with a gain of frequency decay in the process.

\begin{lem}\label{lem:3}
	Let $s,k\in \R$ be such that
	$$
	k-s<3,\quad s+k>-\frac{3}{2},\quad s>-\frac{5}{2}.
	$$
	Then, for $\epsilon<\min\{3-k+s,5+2s\}$,
	\begin{equation}
		\left\| \mathcal{N}_1^u[u_{11},v_{12},v_2]\right\|_{X^{k+\epsilon,b'}}\lesssim \|u_{11}\|_{X^{k,b}}\|v_{12}\|_{Y^{s,b}}\|v_2\|_{Y^{s,b}}.
	\end{equation}	
\end{lem}
\begin{proof}
Recall the resonance function
$$
\Psi=\Psi_1^u=\xi^2-\xi_{11}^2+\xi_{12}^3+\xi_2^3
$$
and the multiplier
$$
\mathcal{M}=\frac{\jap{\xi}^{k+\epsilon}}{|\Phi_1^u|\jap{\xi_{11}}^k\jap{\xi_{12}}^s\jap{\xi_2}^s} \lesssim \frac{\jap{\xi}^{k+\epsilon}}{|\xi_2|^3\jap{\xi_{11}}^k\jap{\xi_{12}}^s\jap{\xi_2}^s}
$$
\vskip10pt
\noindent\textbf{\underline{Case A.}}  $|\xi_{11}|,|\xi_{12}|\ll |\xi_2|$. We interpolate between
$$
\sup_{\xi,\xi_{11}} \int \frac{\jap{\xi}^{k+\epsilon}}{|\xi_2|^{5/2}\jap{\xi_{11}}^k\jap{\xi_{12}}^{s}\jap{\xi_2}^s}\Xi^{0^+} \psia d\xi_{12} \lesssim \sup_{\xi,\xi_{11}} \int \frac{|\xi|^{k+\epsilon-s-3+0^+}}{\jap{\xi_{11}}^k\jap{\xi_{12}}^{s}}\psia \frac{d\Psi}{|\xi_2|^{3/2}} \lesssim M
$$
and
$$
\sup_{\xi_2,\xi_{12}} \int \frac{\jap{\xi}^{k+\epsilon}}{|\xi_2|^{7/2}\jap{\xi_{11}}^{k}\jap{\xi_{12}}^{s}\jap{\xi_2}^s}\Xi^{0^+} \psia d\xi_{11} \lesssim \sup_{\xi_2,\xi_{12}} \int \frac{|\xi|^{k+\epsilon-s-3+0^+}}{\jap{\xi_{11}}^{k}\jap{\xi_{12}}^{s}}\psia \frac{d\Psi}{|\xi_2|^{3/2}} \lesssim M.
$$

\vskip10pt
\noindent\textbf{\underline{Case B.}} $|\xi_{11}|\sim |\xi_{12}|\sim |\xi_2|$. Then we estimate
\begin{align*}
	\sup_{\xi,\xi_2} \int \frac{\jap{\xi}^{k+\epsilon}}{|\xi_2|^3\jap{\xi_{11}}^{k}\jap{\xi_{12}}^{s}\jap{\xi_2}^s}\Xi^{0^+} \psia d\xi_{11} &\lesssim 	\sup_{\xi,\xi_2} \int |\xi|^{\epsilon-5-2s+0^+} \psia d\Psi \lesssim M.
\end{align*}
The other side of the interpolation is completely analogous.

\vskip10pt
\noindent\textbf{\underline{Case C.}} $|\xi_{11}|, |\xi_{12}|\gg |\xi_2|$. Here the intepolation is made between
$$
\sup_{\xi,\xi_{11}} \int \frac{\jap{\xi}^{k+\epsilon}}{|\xi_2|^3 \jap{\xi_{11}}^k\jap{\xi_{12}}^{s-1/2}\jap{\xi_2}^s}\Xi^{0^+}\psia d\xi_{11} \lesssim \sup_{\xi,\xi_{11}} \int \frac{|\xi|^{k+\epsilon-3-s}}{ \jap{\xi_{11}}^{k+s+3/2+0^-}}\psia d\Psi \lesssim M
$$
and
$$
\sup_{\xi_2,\xi_{12}} \int \frac{\jap{\xi}^{k+\epsilon}}{|\xi_2|^3 \jap{\xi_{11}}^k\jap{\xi_{12}}^{s+1/2}\jap{\xi_2}^s}\Xi^{0^+}\psia d\xi_{11} \lesssim \sup_{\xi_2,\xi_{12}} \int \frac{|\xi|^{k+\epsilon-3-s}}{ \jap{\xi_{12}}^{k+s+3/2+0^-}}\psia d\Psi \lesssim M.
$$
\end{proof}

\begin{lem}\label{lem:4}
Let $s,k\in \R$ be such that
$$
k-s<3,\quad k\ge0.
$$
Then, for $\epsilon<3-k+s$,
	\begin{equation}
		\left\| \mathcal{N}_2^u[u_{11},u_{12},u_{13},v_2]\right\|_{X^{k+\epsilon,b'}}\lesssim \|u_{11}\|_{X^{k,b}}\|u_{12}\|_{X^{k,b}}\|u_{13}\|_{X^{k,b}}\|v_2\|_{Y^{s,b}}.
	\end{equation}	
\end{lem}
\begin{proof}
	Here, the resonance function is
	$$
	\Psi=\Psi_2^u=\xi^2-\xi_{11}^2+\xi_{12}^2-\xi_{13}^2+\xi_2^3
	$$
	and the multiplier is given by
	$$
	\mathcal{M}= \frac{\jap{\xi}^{k+\epsilon+1}}{|\Phi_1^u| \jap{\xi_{11}}^k \jap{\xi_{12}}^k \jap{\xi_{13}}^k \jap{\xi_{2}}^s} \lesssim \frac{\jap{\xi}^{k+\epsilon}}{|\xi_2|^3 \jap{\xi_{11}}^k \jap{\xi_{12}}^k \jap{\xi_{13}}^k \jap{\xi_{2}}^s}.
	$$
	\vskip10pt
	\noindent\textbf{\underline{Case A.}} $|\xi_{1j}|\lesssim |\xi_2|$ for $j=1,2,3$. In this case,
	\begin{align*}
		\sup_{\xi,\xi_{11}} \int \frac{\jap{\xi}^{k+\epsilon+1}}{|\xi_2|^3 \jap{\xi_{11}}^k \jap{\xi_{12}}^k \jap{\xi_{13}}^k \jap{\xi_{2}}^s}\Xi^{0^+}\psia d\xi_{12} d\xi_2 \lesssim \sup_{\xi,\xi_{11}} \int |\xi|^{k+\epsilon-s-4+0^+}\psia d\xi_{12} d\Psi \lesssim M
	\end{align*}
	and
	$$
	\sup_{\xi_2,\xi_{12},\xi_{13}} \int \frac{\jap{\xi}^{k+\epsilon-1}}{|\xi_2|^3 \jap{\xi_{11}}^k \jap{\xi_{12}}^k \jap{\xi_{13}}^k \jap{\xi_{2}}^s}\Xi^{0^+}\psia d\xi \lesssim \sup_{\xi_2,\xi_{12},\xi_{13}} \int |\xi|^{k+\epsilon-s-4+0^+} d\xi \lesssim 1,
	$$
	since $k+\epsilon-s-4<-1$.
	\vskip10pt
	\noindent\textbf{\underline{Case B.}} For some $j=1,2,3$. $|\xi_{1j}|\gg |\xi_2|$. Suppose, w.l.o.g., that $|\xi_{11}|\ge|\xi_{12}|\ge |\xi_{13}|$. Then $|\xi_{12}|\gg |\xi_2|\simeq |\xi|$ and we interpolate between
	\begin{align*}
		\sup_{\xi_{11},\xi_{13}} \int \frac{\jap{\xi}^{k+\epsilon}}{|\xi_2|^3 \jap{\xi_{11}}^k \jap{\xi_{12}}^k \jap{\xi_{13}}^k \jap{\xi_{2}}^s}&\Xi^{0^+}\psia d\xi_{12}d\xi_2\\& \lesssim  \sup_{\xi_{11},\xi_{13}}  \int |\xi_2|^{k+\epsilon-s-3}|\xi_{11}|^{0^+}\psia \frac{d\Psi}{|\xi_{11}|}d\xi_{12} \lesssim M
	\end{align*}
	and
	\begin{align*}
		\sup_{\xi,\xi_2\xi_{12}} \int \frac{\jap{\xi}^{k+\epsilon}}{|\xi_2|^3 \jap{\xi_{11}}^k \jap{\xi_{12}}^k \jap{\xi_{13}}^k \jap{\xi_{2}}^s}&\Xi^{0^+}\psia d\xi_{11} \\&\lesssim \sup_{\xi,\xi_2\xi_{12}} \int \jap{\xi}^{k+\epsilon-s-3}|\xi_{11}|^{0^+}\psia \frac{d\Psi}{|\xi_{11}|} \lesssim M.
	\end{align*}
\end{proof}

\begin{lem}\label{lem:5}
	For $k>-1/2$ and $\epsilon<\min\{k+3, 2k+3\}$, 
	\begin{equation}
		\left\| \mathcal{N}_3^u[u_1,u_{21},u_{22}]\right\|_{X^{k+\epsilon,b'}}\lesssim \|u_1\|_{X^{k,b}}\|u_{21}\|_{X^{k,b}}\|u_{22}\|_{X^{k,b}}.
	\end{equation}	
\end{lem}

\begin{proof}
	Recall that the resonance function and the multiplier are  given by
	$$
	\Psi=\Psi_3^u=\xi^2-\xi_1^2-\xi_{21}^2+\xi_{22}^2,\quad \mathcal{M}=\frac{|\xi_2|\jap{\xi}^{k+\epsilon}}{|\Phi_1^u|\jap{\xi_1}^k\jap{\xi_{21}}^k\jap{\xi_{22}}^k}\lesssim \frac{\jap{\xi}^{k+\epsilon}}{\xi_2^2\jap{\xi_1}^k\jap{\xi_{21}}^k\jap{\xi_{22}}^k}.
	$$
\vskip10pt
	\noindent\textbf{\underline{Case A.}} $|\xi_{21}|, |\xi_{22}| \lesssim |\xi|$. By symmetry, assume $|\xi_{22}|\sim |\xi|$. Then
\begin{align*}
		\sup_{\xi,\xi_{21}} \int \frac{\jap{\xi}^{k+\epsilon}}{\xi_2^2\jap{\xi_1}^k\jap{\xi_{21}}^k\jap{\xi_{22}}^k}\Xi^{0^+}\psia d\xi_{22} &\lesssim \sup_{\xi,\xi_{21}} \int \frac{|\xi|^{\epsilon-2+0^+}}{\jap{\xi_1}^k\jap{\xi_{21}}^k}\psia d\xi_{22} \\&\lesssim\sup_{\xi,\xi_{21}} \int \frac{|\xi|^{\epsilon-3+0^+}}{\jap{\xi_1}^k\jap{\xi_{21}}^k}\psia d\Psi \lesssim M.
\end{align*}
	The other side of the interpolation follows from analogous computations.
	
	\vskip10pt
	\noindent\textbf{\underline{Case B.}} $|\xi_{21}| \gg |\xi|\simeq |\xi_2|$, which implies $\xi_{21}\simeq -\xi_{22}$. In this case, the interpolation is made between
	$$
	\sup_{\xi_{21},\xi} \int \frac{\jap{\xi}^{k+\epsilon}}{\xi_2^2\jap{\xi_1}^k\jap{\xi_{21}}^k\jap{\xi_{22}}^k}\Xi^{0^+}\psia d\xi_{22} \lesssim \sup_{\xi_{21},\xi} \int \frac{|\xi|^{k+\epsilon-2}}{\jap{\xi_1}^k|\xi_{21}|^{2k+1^-}}\psia d\Psi \lesssim M
	$$
	and
	$$
	\sup_{\xi_{22},\xi_1} \int \frac{\jap{\xi}^{k+\epsilon}}{\xi_2^2\jap{\xi_1}^k\jap{\xi_{21}}^k\jap{\xi_{22}}^k}\Xi^{0^+}\psia d\xi_{21} \lesssim \sup_{\xi_{22},\xi_1} \int \frac{|\xi|^{k+\epsilon-2}}{\jap{\xi_1}^k|\xi_{22}|^{2k+1^-}}\psia d\Psi \lesssim M.
	$$
\end{proof}

\begin{lem}\label{lem:6}
	Let $s,k\in \R$ be such that
	$$
	k-s<3,\quad s\ge -1.
	$$
	Then, for $\epsilon<3-k+s$,
	\begin{equation}
		\left\| \mathcal{N}_4^u[u_1,v_{21},v_{22}]\right\|_{X^{k+\epsilon,b'}}\lesssim \|u_1\|_{X^{k,b}}\|v_{21}\|_{Y^{s,b}}\|v_{22}\|_{Y^{s,b}}.
	\end{equation}	
\end{lem}
\begin{proof}
	In this case, the resonance function and the multiplier are
	$$
	\Psi=\Psi_4^u=\xi^2-\xi_1^2+\xi_{21}^3+\xi_{22}^3, \quad \mathcal{M}=\frac{\jap{\xi}^{k+\epsilon}|\xi_2|}{|\Phi_1^u|\jap{\xi_1}^k\jap{\xi_{21}}^{s}\jap{\xi_{22}}^{s}}\lesssim \frac{\jap{\xi}^{k+\epsilon}}{|\xi_2|^2\jap{\xi_1}^k\jap{\xi_{21}}^{s}\jap{\xi_{22}}^{s}}.
	$$
	We suppose, w.l.o.g., that $|\xi_{22}|\ge |\xi_{21}|$.
	
		\vskip10pt
		\noindent\textbf{\underline{Case A.}} $|\xi_{22}|\lesssim |\xi|$, which implies $|\xi|\sim |\xi_{22}|\gg |\xi_1|$. We consider two cases:
		\begin{enumerate}
			\item $|2\xi+3\xi_{21}^2|\gtrsim |\xi_{21}|^2$. Then the interpolation is made between
			\begin{align*}
				\sup_{\xi_1,\xi_{22}} \int \frac{\jap{\xi}^{k+\epsilon}}{|\xi_2|^3\jap{\xi_1}^k\jap{\xi_{21}}^{s-1}\jap{\xi_{22}}^{s}}\Xi^{0^+}\psia d\xi_{21} \lesssim 	\sup_{\xi_1,\xi_{22}} \int \frac{|\xi|^{k+\epsilon-s-3+0^+}}{\jap{\xi_1}^k\jap{\xi_{21}}^{s+1}}\psia d\Psi \lesssim M
			\end{align*}
		and
			\begin{align*}
				\sup_{\xi,\xi_{21}} \int \frac{\jap{\xi}^{k+\epsilon}}{|\xi_2|\jap{\xi_1}^k\jap{\xi_{21}}^{s+1}\jap{\xi_{22}}^{s}}\Xi^{0^+}\psia d\xi_{22} \lesssim 	\sup_{\xi,\xi_{21}} \int \frac{|\xi|^{k+\epsilon-s-3+0^+}}{\jap{\xi_1}^k\jap{\xi_{21}}^{s+1}}\psia d\Psi \lesssim M.
			\end{align*}
			\item $|2\xi+3\xi_{21}^2|\ll |\xi_{21}|^2$, which implies $\xi \simeq \xi_{22}$ and $|\xi_{21}|^2\gg |\xi_1|$. In this case, we interpolate between
			\begin{align*}
				\sup_{\xi,\xi_{22}} \int \frac{\jap{\xi}^{k+\epsilon}}{|\xi_2|^3\jap{\xi_1}^k\jap{\xi_{21}}^{s-1}\jap{\xi_{22}}^{s}}\Xi^{0^+}\psia d\xi_{21} \lesssim 	\sup_{\xi,\xi_{22}} \int \frac{|\xi|^{k+\epsilon-s-3+0^+}}{\jap{\xi_{21}}^{s+1}\jap{\xi_1}^k}\psia d\Psi \lesssim M
			\end{align*}
		and
			\begin{align*}
				\sup_{\xi_1,\xi_{21}} \int \frac{\jap{\xi}^{k+\epsilon}}{|\xi_2|\jap{\xi_1}^k\jap{\xi_{21}}^{s+1}\jap{\xi_{22}}^{s}}\Xi^{0^+}\psia d\xi_{22} \lesssim 	\sup_{\xi_1,\xi_{21}} \int \frac{|\xi|^{k+\epsilon-s-3+0^+}}{\jap{\xi_{21}}^{s+1}\jap{\xi_1}^k}\psia d\Psi \lesssim M.
			\end{align*}
		\end{enumerate}
		\vskip10pt
		\noindent\textbf{\underline{Case B.}} $|\xi_{22}|\gg |\xi|$, which implies $\xi_{21}\simeq -\xi_{22}$. Then
\begin{align*}
			\sup_{\xi,\xi_{21}} \int \frac{\jap{\xi}^{k+\epsilon}}{\xi_2^2 \jap{\xi_1}^k\jap{\xi_{21}}^{s}\jap{\xi_{22}}^s}\Xi^{0^+} \psia d\xi_{22} \lesssim
	\sup_{\xi,\xi_{21}} \int \frac{|\xi|^{k+\epsilon-2}}{\jap{\xi_1}^k|\xi_{21}|^{2s+2^-}}\psia d\Psi \lesssim M
\end{align*}
	and the other side of the interpolation follows from similar computations.
\end{proof}

Finally, we prove the multilinear estimates associated with the $N^v_j$ terms. Recall that the integration occurs in the region $|\xi_2|\le|\xi_1|\lesssim |\xi|$.

\begin{lem}\label{lem:7}
	Let $s,k\in \R$ be such that
	$$
	k-s>-\frac{5}{2},\quad s+k>-\frac{3}{2},\quad k\ge 0.
	$$
	Then, for $\epsilon<\min\{5/2+k-s,2k+5/2,4\}$, 
	\begin{equation}
		\left\| \mathcal{N}_1^v[u_{11},v_{12},v_2]\right\|_{Y^{s+\epsilon,b'}}\lesssim \|u_{11}\|_{X^{k,b}}\|v_{12}\|_{Y^{s,b}}\|v_2\|_{Y^{s,b}}.
	\end{equation}	
and
	\begin{equation}
	\left\| \mathcal{N}_3^v[u_{11},v_{12},v_2]\right\|_{Y^{s+\epsilon,b'}}\lesssim \|u_{11}\|_{X^{k,b}}\|v_{12}\|_{Y^{s,b}}\|v_2\|_{Y^{s,b}}.
\end{equation}
\end{lem}
\begin{proof}
We prove only the first estimate, as the second follows from similar arguments. The resonance function is given by
$$
\Psi_1^v=\xi^3-\xi_{11}^2-\xi_{12}^3+\xi_2^2
$$
and the multiplier is
$$
\mathcal{M}=\frac{|\xi|\jap{\xi}^{s+\epsilon}}{|\Phi_1^v|\jap{\xi_{11}}^k\jap{\xi_{12}}^s\jap{\xi_2}^k} \lesssim \frac{\jap{\xi}^{s+\epsilon-2}}{\jap{\xi_{11}}^k\jap{\xi_{12}}^s\jap{\xi_2}^k}.
$$
\vskip10pt
\noindent\textbf{\underline{Case A.}} $|\xi_{12}|\sim |\xi|$. In particular, $|\xi_{11}|\lesssim |\xi|$ and $|\xi_2|\lesssim |\xi_{12}|$. Then
$$
\sup_{\xi, \xi_{11}} \int \frac{|\xi|^{s+\epsilon-2}}{\jap{\xi_{11}}^{k}\jap{\xi_{12}}^s\jap{\xi_2}^k}\Xi^{0^+}\psia d\xi_2 \lesssim \int |\xi|^{\epsilon-4+0^+}\psia d\Psi \lesssim M
$$
and the other side of the interpolation is analogous.

\vskip10pt
\noindent\textbf{\underline{Case B.}} $|\xi_{12}|\ll|\xi|$, which implies  $|\xi_{11}|\sim |\xi|$. Then we interpolate between
$$
\sup_{\xi_{12},\xi_2} \int \frac{|\xi|^{\epsilon-2+3/2}}{\jap{\xi_{11}}^{k}\jap{\xi_{12}}^s\jap{\xi_2}^k}\Xi^{0^+}\psia d\xi_{11}\lesssim \sup_{\xi_{12},\xi_2} \int \frac{|\xi|^{\epsilon+s-k-5/2+0^+}}{\jap{\xi_{12}}^s}\psia d\Psi \lesssim M
$$
and
$$
\sup_{\xi_{11},\xi} \int \frac{|\xi|^{s+\epsilon-2-3/2}}{\jap{\xi_{11}}^{k}\jap{\xi_{12}}^s\jap{\xi_2}^k}\Xi^{0^+}\psia d\xi_2\lesssim \sup_{\xi_{11},\xi} \int \frac{|\xi|^{s+\epsilon-k-7/2+0^+}}{\jap{\xi_{12}}^s} \fia d\Psi\lesssim M.
$$

\vskip10pt
\noindent\textbf{\underline{Case C.}} $|\xi_{12}|\gg |\xi|$ and thus $|\xi_{12}|\sim |\xi_{11}|$. In this region, the interpolation is made between
$$
\sup_{\xi,\xi_{12}} \int \frac{|\xi|^{s+\epsilon-2}}{\jap{\xi_{11}}^{k+1/2}\jap{\xi_{12}}^s\jap{\xi_2}^k}\Xi^{0^+}\psia d\xi_{11} \lesssim \sup_{\xi,\xi_{12}} \int \frac{|\xi|^{s+\epsilon-2}}{|\xi_{11}|^{k+s+(3/2)^-}}\psia d\Psi \lesssim M
$$
and
$$
\sup_{\xi_{11},\xi_2} \int \frac{|\xi|^{s+\epsilon-2}}{\jap{\xi_{11}}^{k-1/2}\jap{\xi_{12}}^s\jap{\xi_2}^k}\Xi^{0^+}\psia d\xi \lesssim \sup_{\xi_{11},\xi_2} \int \frac{|\xi|^{s+\epsilon-2}}{|\xi_{11}|^{k+s+(3/2)^-}}\psia d\Psi \lesssim M.
$$
\end{proof}

\begin{lem}\label{lem:8}
	Let $s,k\in \R$ be such that
	$$
s-k<3,\quad s<2k+5/2,\quad s<4k+1/2,\quad k\ge 0.
	$$
	Then, for $\epsilon<\min\{4k-s+1/2, 3+k-s, 2k-s+5/2\}$,
	\begin{equation}
		\left\| \mathcal{N}_2^v[u_{11},u_{12},u_{13}, u_2]\right\|_{Y^{s+\epsilon,b'}}\lesssim \|u_{11}\|_{X^{k,b}}\|u_{12}\|_{X^{k,b}}\|u_{13}\|_{X^{k,b}}\|u_{2}\|_{X^{k,b}}.
	\end{equation}	
	and
	\begin{equation}
	\left\| \mathcal{N}_4^v[u_{11},u_{12},u_{13}, u_2]\right\|_{Y^{s+\epsilon,b'}}\lesssim \|u_{11}\|_{X^{k,b}}\|u_{12}\|_{X^{k,b}}\|u_{13}\|_{X^{k,b}}\|u_{2}\|_{X^{k,b}}.
\end{equation}	
\end{lem}
\begin{proof}
	Once again, we prove only the first estimate, as the second follows by symmetry. The resonance function is
	$$
	\Psi=\Psi_2^v=-\xi^3-\xi_{11}^2+\xi_{12}^2-\xi_{13}^2+\xi_2^2
	$$
	and the multiplier is given by
	$$
	\mathcal{M}=\frac{|\xi|\jap{\xi}^s}{|\Phi_1^v|\jap{\xi_{11}}^k\jap{\xi_{12}}^k\jap{\xi_{13}}^k\jap{\xi_{2}}^k} \lesssim 	\frac{\jap{\xi}^{s+\epsilon-2}}{\jap{\xi_{11}}^k\jap{\xi_{12}}^k\jap{\xi_{13}}^k\jap{\xi_{2}}^k}
	$$
	We divide the analysis into three cases, depending on the size of $|\xi|$ when compared to the other frequencies.
	\vskip10pt
	\noindent\textbf{\underline{Case A.}} If all frequencies are $\gtrsim |\xi|$, but none is $\gg |\xi|$. Then
	$$
	\mathcal{M}\lesssim \xi^{s+\epsilon -4k-2}.
	$$
	Therefore
	$$
	\sup_{\xi_{13},\xi_2} \int\xi^{s+\epsilon-4k-5/2}d\xi_{11}d\xi \lesssim 1,\quad \sup_{\xi,\xi_{11},\xi_{12}} \int\xi^{s+\epsilon-4k-3/2} d\xi_2 \lesssim 1
	$$
	and we can apply directly Lemma \ref{lem:fre_tri}.
	\vskip10pt
	\noindent\textbf{\underline{Case B.}} All frequencies are $\lesssim |\xi|$ and one is $\ll |\xi|$ (say $|\xi_{13}|$). As $\xi=\xi_{11}+\xi_{12}+\xi_{13}+\xi_2$, there is another frequency, say $\xi_2$, such that $|\xi_{13}|\ll |\xi_2|$. Then
	$$
	\frac{\jap{\xi}^{s+\epsilon}}{\xi^2\jap{\xi_{11}}^k\jap{\xi_{12}}^k\jap{\xi_{13}}^k\jap{\xi_{2}}^k}\lesssim \xi^{s+\epsilon-k-2}
	$$
	and we estimate
	$$
	\sup_{\xi_2,\xi_{13}} \int \xi^{s+\epsilon-k-2}\Xi^{0^+}\psia d\xi d\xi_{11} \lesssim \sup_{\xi_2,\xi_{13}} \int \xi^{s+\epsilon-k-4+0^+}\psia d\Psi d\xi_{11}\lesssim M 
	$$
	and
	$$
	\sup_{\xi,\xi_{11},\xi_{12}} \int \xi^{s+\epsilon-k-2}\Xi^{0^+}\psia d\xi_2 \lesssim \sup_{\xi,\xi_{11},\xi_{12}} \int \xi^{s+\epsilon-k-3+0^+}\psia d\Psi \lesssim M.
	$$
	\vskip10pt
	\noindent\textbf{\underline{Case C.}} There are two frequencies $\gg |\xi|$. As $|\xi|\gtrsim |\xi_2|$, if we order $|\xi_{11}|\ge |\xi_{12}|\ge |\xi_{13}|$, we must have $|\xi_{11}|,|\xi_{12}|\gg |\xi|$. Moreover, the situation $|\xi_{11}|\simeq |\xi_{12}|\simeq |\xi_{13}|$ is impossible: otherwise $|\xi_1|\simeq |\xi_{11}|\gg |\xi|$. Without loss of generality, $|\xi_{11}|\not\simeq |\xi_{13}|$. Then
	$$
	\sup_{\xi_2,\xi_{12}} \int \frac{\xi^{s+\epsilon-5/2}}{\xi_{11}^{2k}}\Xi^{0^+}\psia d\xi_{11}d\xi \lesssim \sup_{\xi_2,\xi_{12}} \int \frac{\xi^{s+\epsilon-5/2}}{\xi_{11}^{2k+1^-}}\psia d\psi d\xi \lesssim M
	$$
	and
	$$
	\sup_{\xi,\xi_{11},\xi_{13}} \int \frac{|\xi|^{s+\epsilon-3/2}}{|\xi_{12}|^{2k}}\Xi^{0^+}\psia d\xi_{12} \lesssim \sup_{\xi_2,\xi_{12}} \int \frac{\xi^{s+\epsilon-3/2}}{\xi_{11}^{2k+1^-}}\psia d\psi \lesssim M.
	$$
\end{proof}

We now move to some estimates for the boundary terms, which we write in physical space as
$$
\mathcal{B}^u[u,v](t) = \mathcal{F}_\xi^{-1}\left(e^{-it\xi^2}B^u[u,v](t)\right),\quad \mathcal{B}^v[u,v](t) = \mathcal{F}_\xi^{-1}\left(e^{-it\xi^3}B^v[u,u]\right).
$$
The first lemma shows the boundedness of the boundary terms in $H^k\times H^s$ for $k,s\in \mathcal{A}$.
\begin{lem}\label{lem:bdryHs}
	For $(k,s)\in \mathcal{A}$,
\begin{equation}\label{eq:est_bordo_u}
	\left\| \mathcal{B}^u[u_1,v_2] \right\|_{L^\infty_t H_x^{k+\epsilon}}\lesssim (\delta^u)^{0^+} \|u_1\|_{L^\infty_t H_x^{k}}\|v_2\|_{L^\infty_t H_x^{s}},\quad \epsilon<\min\{3-k+s,s+5/2\}
\end{equation}
and
\begin{equation}\label{eq:est_bordo_v}
	\left\| \mathcal{B}^v[u_1,u_2] \right\|_{L^\infty_t H_x^{s+\epsilon}}\lesssim(\delta^v)^{0^+} \|u_1\|_{L^\infty_t H_x^{k}}\|u_2\|_{L^\infty_t H_x^{k}},\quad \epsilon<\min\{2-s+k,4k-s\}.
\end{equation}
\end{lem}
\begin{proof}
	By duality, estimate \eqref{eq:est_bordo_u} follows from
	$$
	\left|\int_{U_\xi}\frac{\jap{\xi}^{k+\epsilon}h_1(\xi_1)h_2(\xi_2)h(\xi)}{\Phi^u_1\jap{\xi_1}^k\jap{\xi_2}^s}d\xi d\xi_1\right|\lesssim  (\delta^u)^{0^+} \|h_1\|_{L^2}\|h_2\|_{L^2}\|h\|_{L^2}.
	$$
	Applying Cauchy-Schwarz, this reduces to a bound on
\begin{equation}\label{eq:sup_bordo}
		\sup_{\xi} \left( \int_{U_\xi}\frac{\jap{\xi}^{2k+2\epsilon}}{|\Phi^u_1|^2\jap{\xi_1}^{2k}\jap{\xi_2}^{2s}} d\xi_1 \right)^{\frac{1}{2}}.
\end{equation}
	Over $U_\xi$, $|\Phi_1^u|\simeq |\xi|^3$ and $|\xi_2|\simeq |\xi|>1/\delta^u$, 
	\begin{align*}
			\sup_{\xi}  \int_{U_\xi}\frac{\jap{\xi}^{2k+2\epsilon}}{|\Phi^u_1|^2\jap{\xi_1}^{2k}\jap{\xi_2}^{2s}} d\xi_1  \lesssim 	(\delta^u)^{0^+}\cdot\sup_{\xi} \int_{U_\xi}\frac{\jap{\xi}^{2k+2\epsilon-2s-6+0^+}}{\jap{\xi_1}^{2k}}d\xi_1 .
	\end{align*}
	If $k>1/2$, it suffices to have $k+\epsilon-s-3<0$. If $k<1/2$, the integral is bounded if $s>\epsilon-5/2$.
	
	Analogously, estimate \eqref{eq:est_bordo_v} reduces to a bound on
\begin{equation}\label{eq:sup_bordo1}
		\sup_{\xi}  \int_{V_\xi}\frac{\jap{\xi}^{2s+2\epsilon+2}}{|\Phi^v_1|^2\jap{\xi_1}^{2k}\jap{\xi_2}^{2k}} d\xi_1.
\end{equation}
	As $|\Phi_1^v|\gtrsim |\xi|^3$, $|\xi_1|\gtrsim |\xi|$ and $|\xi|>1/\delta^v$,
	$$
		\sup_{\xi}  \int_{V_\xi}\frac{\jap{\xi}^{2s+2\epsilon+2}}{|\Phi^v_1|^2\jap{\xi_1}^{2k}\jap{\xi_2}^{2k}} d\xi_1\lesssim \sup_\xi \int \frac{\jap{\xi}^{2s+2\epsilon-4}}{\jap{\xi_1}^{2k}\jap{\xi_2}^{2k}} d\xi_1\lesssim (\delta^v)^{0^+}\sup_\xi \int \frac{\jap{\xi}^{2s+2\epsilon-4-2k+0^+}}{\jap{\xi_2}^{2k}} d\xi_2
	$$
		If $k>1/2$, since $s+\epsilon-k-2<0$, the integral is bounded. If $k<1/2$, the integral is finite if $2s+2\epsilon<3+4k$ (which holds since $2s+2\epsilon<8k<3+4k$).
\end{proof}

The next lemma shows how one can control the boundary terms in Bourgain spaces. This can only be achieved for regularities $(k,s)\in \mathcal{A}_0$.
\begin{lem}\label{lem:bordosX}
For $(k,s)\in \mathcal{A}_0$ and $b=(1/2)^+$,
\begin{equation}\label{eq:est_bordo_uX}
	\left\| \mathcal{B}^u[u_1,v_2] \right\|_{X^{k,b}}\lesssim (\delta^u)^{0^+} \|u_1\|_{X^{k,b}}\|v_2\|_{Y^{s,b}},
\end{equation}
 and 
\begin{equation}\label{eq:est_bordo_vX}
	\left\| \mathcal{B}^v[u_1,u_2] \right\|_{Y^{s+\delta, b-\delta/3}}\lesssim(\delta^v)^{0^+} \|u_1\|_{X^{k,b}}\|u_2\|_{X^{k,b}},\quad 0\le\delta\le 1.
\end{equation}
\end{lem}
\begin{proof}
	By duality, the proof of \eqref{eq:est_bordo_uX} reduces to
		$$
	\left| \int_{U_\xi, \tau=\tau_1+\tau_2} \frac{\jap{\tau-\xi^2}^{b}\jap{\xi}^{k}h(\tau,\xi)h_1(\tau_1,\xi_1)h_2(\tau_2,\xi_2)}{|\Phi^u_1|\jap{\tau_1-\xi_1^2}^{b}\jap{\tau_2-\xi_2^3}^{b}\jap{\xi_1}^{k}\jap{\xi_2}^{s}}d\tau d\tau_1 d\xi d\xi_1 \right|\lesssim \|h\|_{L^2}\|h_1\|_{L^2}\|h_2\|_{L^2}.
	$$
	If $|\tau-\xi^2|\lesssim |\tau_1-\xi_1^2|$ (or $|\tau-\xi^2|\lesssim |\tau_2-\xi_2^3|$), by Cauchy-Schwarz, the estimate reduces to
	$$
	\sup_{\xi,\tau} \int_{U_\xi} \left(\frac{\jap{\xi}^{k}}{|\Phi_1^u|\jap{\tau_2-\xi_2^3}^{b}\jap{\xi_1}^{k}\jap{\xi_2}^{s}}\right)^2 d\tau_2 d\xi_2 <\infty.
	$$
	After integrating in $\tau_2$, we find \eqref{eq:sup_bordo} and we are done. If $|\tau-\xi^2|\gg |\tau_1-\xi_1^2|+ |\tau_2-\xi_2^3|$, then $|\tau-\xi^2|\simeq |\Phi_1^u|$. In particular, over this region,
	$$
	\left\| \mathcal{B}^u[u_1,v_2] \right\|_{X^{k,b}} \sim \left\|  u_1v_2 -\mathcal{N}_0^u[u_1,v_2]\right\|_{X^{k,b-1}}\lesssim (\delta^u)^{0^+} \|u_1\|_{X^{k,b}}\|v_2\|_{Y^{s,b}}
	$$
	by Lemma \ref{lem:probU}.
	
\medskip
Similarly, \eqref{eq:est_bordo_vX} is equivalent to 
	$$
\left| \int_{V_\xi, \tau=\tau_1+\tau_2} \frac{\jap{\tau-\xi^3}^{b}|\xi|\jap{\xi}^{s}h(\tau,\xi)h_1(\tau_1,\xi_1)h_2(\tau_2,\xi_2)}{|\Phi^v_1|\jap{\tau_1-\xi_1^2}^{b}\jap{\tau_2-\xi_2^2}^{b}\jap{\xi_1}^{k}\jap{\xi_2}^{k}}d\tau d\tau_1 d\xi d\xi_1 \right|\lesssim \|h\|_{L^2}\|h_1\|_{L^2}\|h_2\|_{L^2}.
$$
If $|\tau-\xi^3|\lesssim |\tau_1-\xi_1^2|$ (or $|\tau-\xi^3|\lesssim |\tau_2-\xi_2^2|$ ), by Cauchy-Schwarz, it suffices to check that
$$
\sup_{\xi,\tau} \int_{V_\xi} \left(\frac{|\xi|\jap{\xi}^{s}}{|\Phi_1^v|\jap{\tau_2-\xi_2^2}^{b}\jap{\xi_1}^{k}\jap{\xi_2}^{k}}\right)^2 d\tau_2 d\xi_2 <\infty.
$$
Integrating in $\tau_2$, this becomes estimate \eqref{eq:sup_bordo1} and the claim follows. If $|\tau-\xi^3|\gg |\tau_1-\xi_1^2|+ |\tau_2-\xi_2^3|$, then $|\tau-\xi^2|\simeq |\Phi_1^v|$ and then
$$
\left\| \mathcal{B}^v[u_1,u_2] \right\|_{Y^{s,b}} \sim \left\|  \partial_x(u_1\overline{u_2}) - \mathcal{N}_0^v[u_1,u_2] \right\|_{X^{s,b-1}}\lesssim (\delta^v)^{0^+} \|u_1\|_{X^{k,b}}\|u_2\|_{X^{k,b}}
$$
by Lemma \ref{lem:probV}.
\end{proof}

The next step involves some smoothing estimates for the terms $\partial_x(v^2)$ and $|u|^2u$. For the latter, we have a standard nonlinear smoothing effect (see also \cite{BPSS}, \cite{CS20}).
\begin{lem}\label{lem:smooth_nls}
	For $k>0$ and $\epsilon<\min\{2k, 1\}$,
	\begin{equation}
		\||u|^2u\|_{X^{k+\epsilon,b'}} \lesssim \|u\|_{X^{k,b}}^3.
	\end{equation}
\end{lem}
\begin{proof}
We apply Lemma \ref{lem:fre_tri1}. The associated multiplier is
$$
\mathcal{M}=\frac{\jap{\xi}^{k+\epsilon}}{\jap{\xi_1}^k \jap{\xi_2}^k\jap{\xi_3}^k}
$$
and the resonance function is $\Phi=\xi^2-\xi_1^2+\xi_2^2-\xi_3^2$. We order the frequencies $|\xi|\ge |\xi_2|$ and $|\xi_1|\ge |\xi_3|$. If all frequencies are smaller than $1$, the frequency-restricted estimates are trivial. Moreover, the worst-case scenario is $|\xi|\ge|\xi_1|$. We then focus on the case $|\xi|>1$.

\smallskip
\noindent\underline{\textbf{Case A.}} $|\xi|\sim |\xi_1|$. Here we consider three subcases:
\begin{enumerate}
	\item $\xi\not\simeq -\xi_2$, which implies that $\xi_1\not\simeq \xi_3$. Then we interpolate between
	$$
	\sup_{\xi,\xi_2} \int \frac{\jap{\xi}^{k+\epsilon}\Xi^{0^+}}{\jap{\xi_1}^k \jap{\xi_2}^k\jap{\xi_3}^k}\fia d\xi_1 \lesssim  	\sup_{\xi,\xi_2} \int \jap{\xi}^{\epsilon+0^+-1}\fia d\Phi \lesssim M
	$$
	and
		$$
	\sup_{\xi_1,\xi_3} \int \frac{\jap{\xi}^{k+\epsilon}\Xi^{0^+}}{\jap{\xi_1}^k \jap{\xi_2}^k\jap{\xi_3}^k}\fia d\xi \lesssim  	\sup_{\xi_!,\xi_3} \int \jap{\xi_1}^{\epsilon+0^+-1}\fia d\Phi \lesssim M.
	$$
	\item $\xi\simeq -\xi_2$ and $|\xi_1|\not\simeq |\xi|$. Then
		$$
	\sup_{\xi,\xi_1} \int \frac{\jap{\xi}^{k+\epsilon}\Xi^{0^+}}{\jap{\xi_1}^k \jap{\xi_2}^k\jap{\xi_3}^k}\fia d\xi_2 \lesssim  	\sup_{\xi,\xi_1} \int \jap{\xi}^{\epsilon+0^+-1}\fia d\Phi \lesssim M
	$$
	and the other side of the interpolation is completely analogous.
	\item $\xi\simeq -\xi_2$ and $|\xi|\simeq |\xi_1|$, which implies that $\xi\simeq \xi_1\simeq -\xi_2\simeq \xi_3$. For $\xi,\xi_2$ fixed, set
	$$
	p_j=\frac{\xi_j}{\xi},\quad q_1=p_1-\frac{1-p_2}{2}, \quad P=1-p_1^2+p_2^2-p_3^2 = p_2-\frac{(1-p_2)^2}{4}-q_1^2 =: c(p_2)-q_1^2.
	$$
	Then
	\begin{align*}
			\sup_{\xi,\xi_2} \int \frac{\jap{\xi}^{k+\epsilon}\Xi^{0^+}}{\jap{\xi_1}^k \jap{\xi_2}^k\jap{\xi_3}^k}\fia d\xi_1 &\lesssim \sup_{\xi,\xi_2} \int |\xi|^{\epsilon+0^+-2k+1}\fia dp_1 \\&\lesssim \sup_{\xi,\xi_2} \int |\xi|^{\epsilon+0^+-2k+1}\mathbbm{1}_{|\xi^2(c(p_2)-q_1^2)-\alpha|<M} \\&\lesssim |\xi|^{\epsilon+0^+-2k}M^{1/2}.
	\end{align*}
For $\xi_1,\xi_3$ fixed, the computation follows the same steps, exchanging the roles of $\xi,\xi_2$ and $\xi_1,\xi_3$.
\end{enumerate}

\smallskip
\noindent\underline{\textbf{Case B.}} $|\xi|\simeq |\xi_2|\gg |\xi_1|\ge |\xi_3|$. The estimate in this case is performed as in Case A.(b). 

\end{proof}

For the $\partial_x(v^2)$ term in the KdV equation, it is well-known that no global smoothing estimate (that is, a gain of derivatives measured in the $H^s$ scale) is possible \cite{imt}. This impediment comes from the $v_\text{high}\times v_\text{low}$ interaction. For this interaction, we cannot gain regularity in frequency, but instead we aim to gain some regularity in the time variable.
To formalize this, we split $N_5^v$ (which corresponds to the term $\partial_x(v^2)$) as
$$
N_{5,\ll}^v[v_1,v_2] = \int_{\xi=\xi_1+\xi_2, |\xi_2|\ll |\xi|\sim |\xi_1|} e^{it\Phi^v_2}\xi \tilde{v}_1\tilde{v}_2d\xi_1, \qquad N_{5,\gtrsim}^v[v_1,v_2] = \int_{\xi=\xi_1+\xi_2, |\xi_1|\gtrsim|\xi_2|\gtrsim |\xi|} e^{it\Phi^v_2}\xi \tilde{v}_1\tilde{v}_2d\xi_1.
$$
\begin{lem}\label{lem:est_dxv2}
For $s>0$,
\begin{equation}\label{eq:est_dxv2_1}
	\|\mathcal{N}_{5,\ll}^v[v_1,v_2]\|_{Y^{s+1,c'}} \lesssim \|v_1\|_{Y^{s+1,b-1/3}}\|v_2\|_{Y^{s,b}},\quad -c'+b>1.
\end{equation}
For $s>1/4$,
\begin{equation}\label{eq:est_dxv2_2}
	\|\mathcal{N}_{5,\gtrsim}^v[v_1,v_2]\|_{Y^{s+1,b'}} \lesssim \|v_1\|_{Y^{s,b}}\|v_2\|_{Y^{s,b}}.
\end{equation}
\end{lem}
\begin{proof}
	\textit{Step 1. Proof of \eqref{eq:est_dxv2_1}.}  Set
	$$
	\sigma_j = \tau_j - \xi_j^3,\quad j=\emptyset, 1, 2.
	$$
	By duality, the estimate is equivalent to
	$$
	\left| \int_{\substack{\xi=\xi_1+\xi_2\\ \sigma=\sigma_1+\sigma_2 + \Phi}} \frac{|\xi|\jap{\xi}^{s+1}\mathbbm{1}_{|\xi_2|\ll |\xi|}}{\jap{\sigma}^{-c'}\jap{\sigma_1}^{b-1/3}\jap{\sigma_2}^{b}\jap{\xi_1}^{s+1}\jap{\xi_2}^{s}}h(\tau,\xi)h_1(\tau_1,\xi_1)h_2(\tau_2,\xi_2)d\tau d\tau_1 d\xi d\xi_1 \right|\lesssim \|h\|_{L^2}\|h_1\|_{L^2}\|h_2\|_{L^2}
	$$
	By Cauchy-Schwarz, the estimate reduces to
	$$
	I:=\sup_{\tau_1,\xi_1} \int \frac{|\xi|^2}{\jap{\sigma}^{-2c'}\jap{\sigma_1}^{2b-2/3}\jap{\sigma_2}^{2b}} d\tau d\xi < \infty.
	$$
	If $|\sigma_1|\ge |\sigma|, |\sigma_2|$, then $|\sigma_1|\gtrsim |\Phi|$.  Performing the change of variables $\xi\mapsto \Phi$,
		$$
	I\lesssim \sup_{\tau_1,\xi_1} \int \frac{1}{\jap{\sigma}^{-2c'}\jap{\Phi}^{2b-2/3}\jap{\sigma_2}^{2b}} d\sigma d\Phi \lesssim \sup_{\tau_1,\xi_1} \int \frac{1}{\jap{\sigma_1+\Phi}^{-2c'}\jap{\Phi}^{2b-2/3}} d\Phi <\infty.
	$$
	If $|\sigma_2|\ge |\sigma|, |\sigma_1|$, then $|\sigma_2|\gtrsim |\Phi|$ and
	$$
	I\lesssim \sup_{\tau_1,\xi_1} \int \frac{1}{\jap{\sigma}^{-2c'}\jap{\sigma_2}^{2b}} d\sigma d\Phi \lesssim \sup_{\tau_1,\xi_1} \int \frac{1}{\jap{\sigma}^{-2c'+2b-1^+}\jap{\Phi}^{1^+}} d\sigma d\Phi <\infty.
	$$
	If $|\sigma|\ge |\sigma_1|, |\sigma_2|$, then $|\sigma|\gtrsim |\Phi|$. By Cauchy-Schwarz, the duality estimate reduces to
	 $$
	 I':=\sup_{\tau,\xi} \int \frac{|\xi|^2}{\jap{\sigma}^{-2c'}\jap{\sigma_1}^{2b-2/3}\jap{\sigma_2}^{2b}} d\tau_1 d\xi_1 < \infty.
	 $$
	Making the change of variables $\xi_1\mapsto\Phi$,
	$$
	I'\lesssim \sup_{\tau,\xi} \int \frac{1}{\jap{\Phi}^{-2c'}\jap{\sigma_1}^{2b-2/3}\jap{\sigma_2}^{2b}} d\sigma_1 d\Phi \lesssim \sup_{\tau,\xi} \int \frac{1}{\jap{\Phi}^{-2c'}\jap{\Phi-\sigma}^{2b-2/3}} d\Phi<\infty.
	$$
	\smallskip
	\noindent\textit{Step 2. Proof of \eqref{eq:est_dxv2_2}.} We are going to apply \eqref{eq:frequad}, taking $1<M\lesssim |\alpha|$. 
	
	We first consider the case $|\xi|\ge |\xi_1| \ge |\xi_2|$. Since $|\xi_2|\gtrsim |\xi|$, this implies that $|\Phi|\sim |\xi|^3$, $|\xi|\not\simeq |\xi_1|$ and $|\xi|\not\simeq |\xi_2|$. As such,
	$$
	\sup_{\xi_1} \int \frac{|\xi|^2\jap{\xi}^{2s +2+ 0^+}}{\jap{\xi_1}^{2s}\jap{\xi_2}^{2s}} \fia d\xi \lesssim 	\sup_{\xi_1} \int {\jap{\xi}^{2-2s + 0^+}} \fia d\Phi \lesssim |\alpha|^{(2-2s+0^+)/3}M.
	$$
	The estimate is the same if one fixes $\xi_2$ instead of $\xi_1$.
	
	Now we consider \eqref{eq:frequad} with $\xi$ fixed. If $ |\xi_1|\not\simeq |\xi_2|$, then we may once again perform the change of variables $\xi_1\mapsto \Phi$:
		$$
	\sup_{\xi} \int \frac{|\xi|^2\jap{\xi}^{2s +2+ 0^+}}{\jap{\xi_1}^{2s}\jap{\xi_2}^{2s}} \fia d\xi_1 \lesssim 	\sup_{\xi} \int {\jap{\xi}^{2-2s + 0^+}} \fia d\Phi \lesssim |\alpha|^{(2-2s+0^+)/3}M.
	$$
	If $ |\xi_1|\simeq |\xi_2|$, then we normalize the frequencies
	$$
	p_j=\frac{\xi_j}{\xi},\quad P=1-p_1^3-p_2^3.
	$$
	As we're near the stationary point $p_1=1/2$, by Morse's Lemma, there exists a change of coordinates $p_1\mapsto q$ so that
	$$
	P=\frac{3}{4} - q^2,\quad |q|\ll 1.
	$$
	Then
\begin{align*}
	\sup_{\xi} \int \frac{|\xi|^2\jap{\xi}^{2s +2+ 0^+}}{\jap{\xi_1}^{2s}\jap{\xi_2}^{2s}} \fia d\xi_1 &\lesssim 	\sup_{\xi} \int {\jap{\xi}^{5-2s + 0^+}} \fia dp_1 \\&\lesssim 	\sup_{\xi} \int {\jap{\xi}^{5-2s + 0^+}} \mathbbm{1}_{|\xi^3(3/4-q^2)-\alpha|<M} dq \\&\lesssim |\alpha|^{7/6^+-2s/3}M^{1/2}\lesssim |\alpha|M^{1/2}
\end{align*}
since $s>1/4$.

Finally, we consider the case $|\xi|\ll |\xi_1|\simeq |\xi_2|$. Then
	$$
\sup_{\xi_1} \int \frac{|\xi|^2\jap{\xi}^{2s +2+ 0^+}}{\jap{\xi_1}^{2s}\jap{\xi_2}^{2s}} \fia d\xi \lesssim 	\sup_{\xi_1} \int {|\xi|^2\jap{\xi_2}^{-2s + 0^+}} \fia d\Phi \lesssim |\alpha|^{(2-2s+0^+)/3}M,
$$
the estimate for $\xi_2$ fixed is the same, and
	$$
\sup_{\xi} \int \frac{|\xi|^2\jap{\xi}^{2s +2+ 0^+}}{\jap{\xi_1}^{2s}\jap{\xi_2}^{2s}} \fia d\xi_1 \lesssim 	\sup_{\xi} \int {|\xi|\jap{\xi_2}^{1-2s + 0^+}} \fia d\Phi \lesssim |\alpha|^{(2-2s+0^+)/3}M.
$$
Applying \eqref{eq:frequad}, this concludes the proof.
\end{proof}

We conclude this section with a multilinear estimate which will be essential for the local existence result.

\begin{lem}\label{lem:N0B}
	Let
	$$
	\mathcal{K}[u_{11}, v_{12}, v_2]= \mathcal{F}_\xi^{-1}\left(e^{-it\xi^2}N_0^u[B^u[u_{11}, v_{12}], v_2]\right).
	$$
	Then, for $(k,s)\in \mathcal{A}$ and $k\ge 1$,
	$$
	\|\mathcal{K}[u_{11}, v_{12}, v_2]\|_{X^{k,b'}} \lesssim \|u_{11}\|_{X^{k-1,b}}\|v_{12}\|_{Y^{s,b}}\|v_2\|_{Y^{s,b}}.
	$$
\end{lem}
\begin{proof}
	We apply Lemma \ref{lem:fre_tri1}. The associated multiplier is controlled by
	$$
	\mathcal{M}=\frac{\jap{\xi}^k}{|\Phi^u_1|\jap{\xi_{11}}^{k-1}\jap{\xi_{12}}^s \jap{\xi_2}^s} \lesssim \frac{\jap{\xi}^k}{\jap{\xi_{12}}^{s+3}\jap{\xi_2}^s}
	$$
	and the resonance function is given by
	$$
	\Phi=\xi^2-\xi_{11}^2 - \xi_{12}^3 - \xi_2^3.
	$$
	Due to the restrictions in both $N_0^u$ and $B^u$, either all frequencies are bounded (in which case the result is trivial) or
	$$
	|\xi_{11}|\ll |\xi_{12}| \simeq |\xi_1|,\quad |\xi_1|\gtrsim |\xi_2|.
	$$
	We split into sveral cases, in order to ensure that the change of variables to $\Phi$ is large.

	\smallskip\noindent\underline{\textbf{Case A.}} $|\xi|\gtrsim |\xi_{12}| \gtrsim |\xi_2|$ and $3\xi_2^2 \not\simeq 2|\xi|$. Then we interpolate between
	$$
	\sup_{\xi,\xi_2} \int \frac{\jap{\xi}^{k+0^+}}{\jap{\xi_{12}}^{s+3}\jap{\xi_2}^s}\fia d\xi_{11} \lesssim \sup_{\xi,\xi_2} \int \xi^{k-2s-5+0^+} \fia d\Phi \lesssim M
	$$
	and
		$$
	\sup_{\xi_{11},\xi_{12}} \int \frac{\jap{\xi}^{k+0^+}}{\jap{\xi_{12}}^{s+3}\jap{\xi_2}^s}\fia d\xi_{2} \lesssim \sup_{\xi,\xi_2} \int \xi^{k-2s-4+0^+} \fia d\Phi \lesssim M.
	$$
	
	\smallskip\noindent\underline{\textbf{Case B.}} $|\xi|\gtrsim |\xi_{12}|$ and $3\xi_2^2\simeq 2|\xi|$. Then $|\xi_2|\not\simeq |\xi_{12}|$ and we interpolate between
	$$
	\sup_{\xi,\xi_{11}} \int \frac{\jap{\xi}^{k+0^+}}{\jap{\xi_{12}}^{s+3}\jap{\xi_2}^s}\fia d\xi_{12} \lesssim \sup_{\xi,\xi_{11}} \int \xi^{k-2s-5+0^+} \fia d\Phi \lesssim M
	$$
	and
	$$
	\sup_{\xi_{2},\xi_{12}} \int \frac{\jap{\xi}^{k+0^+}}{\jap{\xi_{12}}^{s+3}\jap{\xi_2}^s}\fia d\xi \lesssim \sup_{\xi,\xi_2} \int \xi^{k-2s-4+0^+} \fia d\Phi \lesssim M.
	$$
	
	\smallskip\noindent~\underline{\textbf{Case C.}} $|\xi_2|\simeq |\xi_{12}|\gg |\xi|$. Then
		$$
	\sup_{\xi,\xi_2} \int \frac{\jap{\xi}^{k+0^+}}{\jap{\xi_{12}}^{s+3}\jap{\xi_2}^s}\fia d\xi_{11} \lesssim \sup_{\xi,\xi_2} \int \xi_{2}^{k-2s-5+0^+} \fia d\Phi \lesssim M
	$$
	and
	$$
	\sup_{\xi_{11},\xi_{12}} \int \frac{\jap{\xi}^{k+0^+}}{\jap{\xi_{12}}^{s+3}\jap{\xi_2}^s}\fia d\xi_{2} \lesssim \sup_{\xi,\xi_2} \int \xi_{12}^{k-2s-5+0^+} \fia d\Phi \lesssim M.
	$$
\end{proof}

\section{Local and Global Well-Posedness}
\label{sec:well}

\begin{proof}[Proof of Theorem \ref{thm:well2}]
	We split the proof in the cases
	$$
	k-s<2,\quad -1< k-s.
	$$
	We recall that, in the first strip, the obstacle for local well-posedness lies in the term $\partial_x(|u|^2)$ in the equation for $v$. On the second strip, the impediment resides in the term $uv$ in the equation for $u$. To lighten the proof, we take $\alpha=\beta=\gamma=1$, as it is clear that these constants do not impact the argument in any meaningful way.
	
	\medskip
	
	\noindent \underline{\textbf{Case 1.}} $(k,s)\in\mathcal{A}$ with $-2<k-s\le - 1$.
	
	\smallskip\noindent
	\textit{Step 1. Fixed-point at low regularity.} Choose $s-1\le s'\le s$ so that $s'\ge 1$ and $(k,s')\in \mathcal{A}_0$.
	Consider $\eta\in C^\infty_c(\R)$ with $\eta\equiv 1$ in $[-1,1]$ and set $\eta_T(t)=\eta(t/T)$. Given $R>0$, define
	$$
	B_R:=\{(u,v)\in X^{k,b}\times (Y^{s,b-1/3}\cap Y^{s',b}): \|u\|_{X^{k,b}}+ \|v\|_{Y^{s,b-1/3}} + \|v\|_{Y^{s',b}}<R\},
	$$
	endowed with the induced metric, and consider the mapping $\Theta[u,v]=(w,z)$ defined as
	\begin{align}
		w(t)&=\eta(t)\left( e^{it\partial^2_x}u_0 -i \int_0^t \eta_T(s) e^{i(t-s)\partial_x^2}(uv + |u|^2u)(s)ds \right)\\
		z(t)&=\eta(t)\left( e^{-t\partial^3_x}v_0 - i[\mathcal{B}^v[u,u](s)]_{s=0}^{s=t} - i \sum_{j=0}^5 \int_0^t\eta_T(s) e^{-(t-s)\partial_x^3}\mathcal{N}_j^v(s)ds\right)
	\end{align}
Applying the known linear Bourgain space estimates and Lemmas \ref{lem:est_nls} and \ref{lem:est_wu} for the nonlinear terms, we find
\begin{align*}
	\|w\|_{X^{k,b}} &\le C \left(\|u_0\|_{H^k} + T^{0^+}\cdot \|u\|_{X^{k,b}}\|v\|_{Y^{s',b}} + T^{0^+} \|u\|_{X^{k,b}}^3 \right)\\&\le C\|u_0\|_{H^k}  +  CT^{0^+}R^2(1+R).
\end{align*}
Concerning the equation for $z$, Lemma \ref{lem:est_dxv2} gives
$$
\left\|\int_0^t\eta_T(s) e^{-(t-s)\partial_x^3}\mathcal{N}_5^v(s)ds \right\|_{Y^{s,b-1/3}} \lesssim T^{0^+}\|v\|_{Y^{s,b-1/3}} \|v\|_{Y^{s',b}} \lesssim T^{0^+}R^2 
$$
while Lemma \ref{lem:est_kdv} implies 
$$
\left\|\int_0^t\eta_T(s) e^{-(t-s)\partial_x^3}\mathcal{N}_5^v(s)ds \right\|_{Y^{s',b}} \lesssim T^{0^+}\|v\|_{Y^{s',b}} \|v\|_{Y^{s',b}} \lesssim T^{0^+}R^2.
$$
On the other hand, using the estimates of Lemmas \ref{lem:2}, \ref{lem:7} and \ref{lem:8}, 
\begin{align}
	\left\|\sum_{j=0}^4 \int_0^t\eta_T(s) e^{-(t-s)\partial_x^3}\mathcal{N}_j^v(s)ds \right\|_{Y^{s,b}} &\lesssim C_{\delta^v}T^{0^+}(\|u\|_{X^{k,b}}^2 + \|u\|_{X^{k,b}}^4 + \|u\|_{X^{k,b}}^2 \|v\|_{Y^{s',b}})\\& \lesssim C_{\delta^v}T^{0^+}R^2(1+R^2).\label{eq:est_ponto_fixo_v}
\end{align}
Together with \eqref{eq:est_bordo_vX},
\begin{align*}
	\|z\|_{Y^{s,b-1/3}} + \|z\|_{Y^{s',b}} \le C\left( \|v_0\|_{H^s} + (\delta^v)^{0^+}\|u\|_{X^{k,b}}^2 + T^{0^+}R^2(1+R^2) \right)
\end{align*}
We choose $R>2C(\|u_0\|_{H^k}+\|v_0\|_{H^s})$ and then we take $\delta^v>0$ so that
$$
(\delta^v)^{0^+}R^2 \le \frac{R}{10C}.
$$
Taking $T$ sufficiently small, we conclude that $\Theta:B_R\to B_R$. In a similar fashion, one may prove that $\Theta$ is a strict contraction over $B_R$. As such, we conclude that $\Theta$ has a fixed-point
$$
(u,v)\in X^{k,b}\times (Y^{s,b-1/3}\cap Y^{s',b}).
$$

\noindent\textit{Step 2. Upgrade in regularity.} Now we show that the fixed-point obtained in Step 1 satisfies
$$
(u,v)\in C_tH^k_x \times C_tH^s_x.
$$
First, since $u\in X^{k,b}\hookrightarrow  C_tH^k_x $, Lemma \ref{lem:bdryHs} shows that $\mathcal{B}^v[u,u]\in C_tH^s_x$. By \eqref{eq:est_ponto_fixo_v}, 
$$
\sum_{j=0}^4 \int_0^t\eta_T(s) e^{-(t-s)\partial_x^3}\mathcal{N}_j^v(s)ds \in Y^{s,b}\hookrightarrow C_tH^s_x.
$$
Taking $1/2<c<b$, Lemma \ref{lem:est_dxv2} implies that
$$
\int_0^t\eta_T(s) e^{-(t-s)\partial_x^3}\mathcal{N}_5^v(s)ds \in Y^{s,c}\hookrightarrow C_tH^s_x.
$$
As such, since
$$
	v(t)=\eta(t)\left( e^{-t\partial^3_x}v_0 - i[\mathcal{B}^v[u,u](s)]_{s=0}^{s=t} - i \sum_{j=0}^5 \int_0^t\eta_T(s) e^{-(t-s)\partial_x^3}\mathcal{N}_j^v(s)ds\right),
$$
we conclude that $v\in C_tH^s_x$, as claimed. The continuous dependence on initial data follows from the fixed-point argument and the multilinear bounds used throughout this step.

	\medskip

\noindent \underline{\textbf{Case 2.}} $(k,s)\in\mathcal{A}$ with $2\le k-s< 3$. 

\smallskip\noindent
\textit{Step 1. Fixed-point at low regularity.}
Choose $k-1<k'\le k$ so that $(k',s)\in \mathcal{A}_0$.  Given $R>0$, define
$$
B_R:=\{(u,v)\in X^{k',b}\times Y^{s,b}: \|u\|_{X^{k',b}}+  \|v\|_{Y^{s,b}}<R\},
$$
endowed with the induced metric, and consider the mapping $\Theta[u,v]=(w,z)$ defined as
\begin{align}
	w(t)&=\eta(t)\left( e^{it\partial^2_x}u_0 - i[\mathcal{B}^u[u,v](s)]_{s=0}^{s=t} - i \sum_{j=0}^5 \int_0^t \eta_T(s) e^{i(t-s)\partial_x^2}\mathcal{N}_j^u(s)ds\right)\\
	z(t)&=\eta(t)\left( e^{-t\partial^3_x}v_0 + \int_0^t\eta_T(s) e^{-(t-s)\partial_x^3}(\partial_x(|u|^2) + \partial_x(v^2))(s)ds\right)
\end{align}
Applying Lemmas \ref{lem:est_kdv} and \ref{lem:est_wu},
\begin{align*}
	\|z\|_{X^{k,b}} &\le C \left(\|v_0\|_{H^s} + T^{0^+}\cdot (\|u\|_{X^{k',b}}^2+\|v\|_{Y^{s,b}}^2) \right)\\&\le C\|v_0\|_{H^s}  +  CT^{0^+}R^2.
\end{align*}
We now bound the equation for $w$. By Lemma \ref{lem:smooth_nls},
\begin{align}\label{eq:ponto_fixo_u1}
	\left\|\int_0^t \eta_T(s) e^{i(t-s)\partial_x^2} \mathcal{N}_5^u(s)ds \right\|_{X^{k,b}} \lesssim T^{0^+} \|u\|_{X^{k',b}}^3 \lesssim T^{0^+}R^3.
\end{align}
Using the estimates in Lemmas \ref{lem:1} and \ref{lem:3}-\ref{lem:6},
\begin{align}\label{eq:ponto_fixo_u2}
	\left\|\sum_{j=1}^4\int_0^t \eta_T(s) e^{i(t-s)\partial_x^2} \mathcal{N}_j^u(s)ds \right\|_{X^{k,b}} &\lesssim T^{0^+}\|u\|_{X^{k',b}}\left( \|u\|_{X^{k',b}}^4 + \|v\|_{X^{s,b}}^2 \right)\\&\lesssim T^{0^+}R^3(1+R^2).
\end{align}
and
\begin{align*}
	\left\|\int_0^t \eta_T(s) e^{i(t-s)\partial_x^2} \mathcal{N}_0^u(s)ds \right\|_{X^{k',b}} \lesssim T^{0^+}\|u\|_{X^{k',b}} \|v\|_{X^{s,b}} \lesssim T^{0^+}R^2.
\end{align*}
Together with \eqref{eq:est_bordo_uX},
\begin{align*}
	\|w\|_{X^{k',b}} \le C(\|u_0\|_{H^k} + (\delta^u)^{0^+}\|u\|_{X^{k',b}}\|v\|_{Y^{s,b}} + T^{0^+}R^2(1+R^3)).
\end{align*}
As before, for $R>2C(\|u_0\|_{H^k}+\|v_0\|_{H^s})$, take $\delta^u>0$ so that
$$
(\delta^u)^{0^+}R^2 \le \frac{R}{10}.
$$
For $T$ small, we conclude that $\Theta:B_R\to B_R$. Analogously, it is easy to see that $\Theta$ is a strict contraction over $B_R$. In particular, $\Theta$ has a fixed-point
$$
(u,v)\in X^{k',b}\times Y^{s,b}.
$$
\smallskip\noindent
\textit{Step 2. Upgrade in regularity.} We now show that
$$
(u,v)\in C_t H^k_x \times C_t H^s_x.
$$
The regularity in $v$ comes directly from $v\in Y^{s,b}\hookrightarrow C_tH^s_x$. By \eqref{eq:ponto_fixo_u1} and \eqref{eq:ponto_fixo_u2},
$$
\sum_{j=1}^4\int_0^t \eta_T(s) e^{i(t-s)\partial_x^2} \mathcal{N}_j^u(s)ds \in X^{k,b}\hookrightarrow C_tH^k_x.
$$
By Lemma \ref{lem:bdryHs}, $\eta_T(t)\mathcal{B}^u[u,v](0)\in X^{k,b}$, as it corresponds to the linear evolution of an element in $H^k$. Since
$$
		u(t)=\eta(t)\left( e^{it\partial^2_x}u_0 - i[\mathcal{B}^u[u,v](s)]_{s=0}^{s=t} - i \sum_{j=0}^5 \int_0^t \eta_T(s) e^{i(t-s)\partial_x^2}\mathcal{N}_j^u(s)ds\right),
$$
we find that
$$
	u(t) - i\mathcal{B}^u[u,v](t) + i \int_0^t \eta_T(s) e^{i(t-s)\partial_x^2}\mathcal{N}_0^u[u,v](s)ds \in X^{k,b}.
$$
Define $r(t)= u(t) - i\mathcal{B}^u[u,v](t)$, so that
$$
r(t) + i \int_0^t \eta_T(s) e^{i(t-s)\partial_x^2}\mathcal{N}_0^u[r+i\mathcal{B}^u[u,v],v](s)ds \in X^{k,b}.
$$
Applying Lemma \ref{lem:N0B},
$$
L_v[r] := r(t) + i \int_0^t \eta_T(s) e^{i(t-s)\partial_x^2}\mathcal{N}_0^u[r,v](s)ds = \int_0^t \eta_T(s) e^{i(t-s)\partial_x^2}\mathcal{N}_0^u[\mathcal{B}^u[u,v],v](s)ds \in X^{k,b}.
$$
By Lemma \ref{lem:1},
$$
\|L_v[f]-f\|_{X^{k,b}} \le  CT^{0^+}\|v\|_{Y^{s,b}}\|f\|_{X^{k,b}} < \frac{1}{2}\|f\|_{X^{k,b}}
$$
which implies that $L_v$ is an invertible map in $X^{k,b}$. In particular,
$$
r(t)= u(t) - i\mathcal{B}^u[u,v](t) \in X^{k,b}\hookrightarrow C_tH^k_x.
$$
Finally, by Lemma \ref{lem:bdryHs}, the map $u\mapsto r$ is an invertible map in $C_tH^k_x$ and therefore $u\in C_tH^k_x$, as claimed.

\smallskip\noindent
\underline{\textbf{Case 3.}} $(k,s)\in \mathcal{A}_0$. In this case, using Lemma \ref{lem:bordosX} and arguing as before, we can show that the mapping $\Theta$ (of either Case 1 or 2) is a strict contraction in 
$$
B_R:=\{(u,v)\in X^{k,b}\times Y^{s,b}: \|u\|_{X^{k,b}}+  \|v\|_{Y^{s,b}}<R\},
$$
which concludes the proof.
\end{proof}
\begin{nb}
	One may wonder whether it is possible to perform a direct fixed-point argument at regularities $(k,s)\in \mathcal{A}\setminus \mathcal{A}_0$. Using the estimates of Section \ref{sec:multi}, the time-integral terms are all controllable in $X^{k,b}\times Y^{s,b}$. The problem resides in the boundary terms, which can be bounded in the Bourgain space only for regularities $(k,s)\in \mathcal{A}_0$ (Lemma \ref{lem:bordosX}), \textit{despite being controllable in }$H^k\times H^s$ \textit{for} $(k,s)\in \mathcal{A}$ (Lemma \ref{lem:bdryHs}).
	
	In other words, the integration-by-parts highlights the failure of a straightforward $X^{k,b}\times Y^{s,b}$ approach, as this auxiliary space ceases to be well-suited to build the solution to \eqref{nlskdv}. Instead, the above proof shows that the solution is the sum of an element in $X^{k,b}\times Y^{s,b}$ (the integral terms) and the boundary terms which are merely in $C_tH_x^k\times C_tH_x^s$.
\end{nb}

\begin{prop}[Persistence of regularity] \label{prop:persist} Let $(k,s), (k',s')\in \mathcal{A}$ with $k'\ge k$, $s'\ge s$ and $(u_0,v_0)\in H^{k'}(\R)\times H^{s'}(\R)$. Consider the maximal solutions $(u,v)$ and $(u',v')$ given by Theorem \ref{thm:well2} at the regularity levels $(k,s)$ and $(k',s')$, respectively. Then the maximal time of existence coincides and $(u,v)=(u',v')$.
\end{prop}

\begin{proof}
\textit{Step 1. Local persistence in $\mathcal{A}_0$.} Take $\eta_T$ as in the proof of Theorem \ref{thm:well2}. Let us first recall the local existence result obtained by Wu in \cite{wu}: given $(k,s)\in \mathcal{A}_0$ and $(u_0,v_0)\in H^k\times H^s$, there exists $T>0$, depending only on the norm of the initial data, and a unique
$$
(u,v)\in X^{k,b}\times Y^{s,b}
$$
(depending continuously on $(u_0,v_0)$) such that 
\begin{align}
	u(t)&=\eta(t)\left( e^{it\partial^2_x}u_0 -i \int_0^t \eta_T(s) e^{i(t-s)\partial_x^2}(uv + |u|^2u)(s)ds \right)\\\\
	v(t)&=\eta(t)\left( e^{-t\partial^3_x}v_0 + \int_0^t\eta_T(s) e^{-(t-s)\partial_x^3}(\partial_x(|u|^2) + \partial_x(v^2))(s)ds\right).
\end{align}
\medskip

Take $(k,s)\in\mathcal{A}_0$ and $\epsilon$ satisfying the conditions of Lemmas \ref{lem:1} to \ref{lem:probV}. For fixed $k\le k'\le k'+\epsilon$ and $s\le s'\le s+\epsilon$, suppose that there exists an initial data $(u_0,v_0)\in H^{k'}\times H^{s'}$ for which the maximal time of existence in $H^{k'}\times H^{s'}$, $T_{max}'$, is strictly smaller than the maximal time in $H^k\times H^s$, $T_{max}$. Since
$$
\sup_{t\in [0,T_{max}']} \|(u(t),v(t))\|_{H^k\times H^s} \le R < \infty,
$$
after a possible time translation, we can assume that $T_{max}'$ is smaller than the local time of existence $T$ for initial data whose $H^k\times H^s$ norm is smaller than $R$. By uniqueness, the $H^{k'}\times H^{s'}$ solution coincides on $[0,T_{max}')$ with the $H^k\times H^s$ local solution, for which one has
$$
\|u\|_{X^{k,b}} + \|v\|_{Y^{s,b}} \lesssim R.
$$
Applying the nonlinear smoothing estimates of Lemmas \eqref{lem:1} to \eqref{lem:probV},
\begin{align*}
	\left\|\int_0^t \eta_T(s) e^{i(t-s)\partial_x^2}(uv + |u|^2u)(s)ds\right\|_{X^{k',b}}\lesssim T^{0^+}\|u\|_{X^{k,b}}\left(\|u\|_{X^{k,b}}^2+ \|v\|_{Y^{s,b}}^2\right) \lesssim T^{0^+}R^2(1+R)
\end{align*}
and
\begin{align*}
	\left\|\int_0^t\eta_T(s) e^{-(t-s)\partial_x^3}(\partial_x(|u|^2))(s)ds\right\|_{Y^{s',b}}\lesssim T^{0^+}\|u\|_{X^{k,b}}^2 \lesssim T^{0^+}R^2.
\end{align*}
Moreover, Lemma \ref{lem:est_kdv} implies that
\begin{align*}
	\left\|\int_0^t\eta_T(s) e^{-(t-s)\partial_x^3}(\partial_x(v^2))(s)ds\right\|_{Y^{s',b}}\lesssim T^{0^+}\|v\|_{Y^{s,b}}\|v\|_{Y^{s',b}} \lesssim T^{0^+}R\|v\|_{Y^{s',b}}.
\end{align*}
Therefore,
$$
\|u\|_{X^{k',b}} + \|v\|_{Y^{s',b}} \lesssim \|u_0\|_{H^{k'}} + \|v_0\|_{H^{s'}} + T^{0^+}R^2(1+R) + T^{0^+}R\|v\|_{Y^{s',b}}.
$$
Since $T$ can be made small depending on $R$, the last term can be absorbed into the left-hand side and thus, for all $t<T'_{max}$,
$$
\|u(t)\|_{H^{k'}} + \|v(t)\|_{H^{s'}}\lesssim \|u\|_{X^{k',b}} + \|v\|_{Y^{s',b}} \lesssim \|u_0\|_{H^{k'}} + \|v_0\|_{H^{s'}} +  T^{0^+}R^2(1+R) <\infty.
$$
This is a contradiction with the blow-up alternative and the persistence property holds.

\smallskip
\textit{Step 2. Global persistence in $\mathcal{A}_0$.} Given any two regularities $(k,s)$ and $(k',s')$ in $\mathcal{A}_0$ with $k'\ge k$ and $s'\ge s$, one may find a uniform $\epsilon$ for which the nonlinear smoothing estimates hold for all the regularities in the segment between $(k,s)$ and $(k',s')$. Therefore the (finite) iteration of the local persistence property from Step 1 yields the persistence property for  $(k,s)$ and $(k',s')$.

\smallskip
\textit{Step 3. Global persistence in $\mathcal{A}$.} Take regularities $(k,s)$ and $(k',s')$ in $\mathcal{A}$ such that $k'\ge k$ and $s'\ge s$.

If $(k,s)\notin \mathcal{A}_0$, the proof of Theorem \ref{thm:well2} implies that there exists a regularity level $(k_0,s_0)\in \mathcal{A}_0$ with $k_0\le k$ and $s_0\le s$ such that the persistence property holds between $(k,s)$ and $(k_0,s_0)$ (that is, the maximal time of existence at both regularities coincides). If $(k,s)\in \mathcal{A}_0$, this holds for the trivial regularity $(k_0,s_0)=(k,s)$. 

Analogously, we may define $(k_0',s_0')\in \mathcal{A}_0$ so that persistence holds between $(k',s')$ and $(k'_0,s'_0)$. Moreover, one may check that $k'_0\ge k_0$ and $s'_0\ge s_0$.

In particular, by the previous step, the persistence property holds for the pair $(k_0,s_0)$ and $(k'_0,s'_0)$ and therefore the maximal time of existence in all four regularity levels must coincide.
\end{proof}

\begin{proof}[Proof of Theorem \ref{thm:gwp}]
It suffices to combine the global existence result for $k=s>1/2$ \cite{wu} with Proposition \ref{prop:persist}
\end{proof}

\section{Ill-posedness}\label{sec:ill}

In this section, we prove Theorem \ref{thm:illposed}. First, observe that the known ill-posedness results for (NLS) below $L^2(\R)$ \cite{christcolliandertao} and (KdV) below $H^{-3/4}(\R)$ \cite{bou_ill_kdv} are immediately extendible to system \eqref{nlskdv} by considering initial data of the form $(u_0,0)$ or $(0,v_0)$. Furthermore, it was proven in \cite{wu} that \eqref{nlskdv} is ill-posed for $s>4k$. As such, we are left with proving ill-posedness in the regions
$$
\mathcal{I}_1=\left\{(k,s)\in \R^2: k>0,\quad  s>-\frac{3}{4},\quad s<4k,\quad k-s>3\right\}
$$
and
$$
\mathcal{I}_2=\left\{(k,s)\in \R^2: k>0,\quad  s>-\frac{3}{4},\quad s<4k,\quad k-s<-2\right\}.
$$

\begin{lem}
	System \eqref{nlskdv} is $C^2$-ill-posed in $H^k(\R)\times H^s(\R)$ for $(k,s)\in \mathcal{I}_1$.
\end{lem}
\begin{proof}
For a fixed $N\gg 1$, take
$$
\hat{u}_0=\mathbbm{1}_{[-1,1]},\quad \hat{v}_0=\mathbbm{1}_{[N-1, N+1]}.
$$
We have $\|u_0\|_{H^k}\sim 1$ and $\|v_0\|_{H^s}\sim N^s$. Assuming that the flow is $C^2$ at the origin, if the initial data is of the form $(\epsilon u_0, \epsilon v_0)$, then the $\epsilon^2$-term for $u$ is
$$
I[u_0,v_0]:=\int_0^t e^{i(t-s)\partial_x^2}(e^{is\partial_x^2}u_0 \cdot e^{-s\partial_x^3}v_0)ds
$$
which writes, in frequency variables, as
$$
\int_0^t \int e^{is\Phi}\hat{u}_0(\xi_1)\hat{v}_0(\xi_2)d\xi_1ds = \int \frac{e^{it\Phi}-1}{i\Phi}\hat{u}_0(\xi_1)\hat{v}_0(\xi_2)d\xi_1
$$
(notice that this is exactly the boundary term found in the integration by parts in time). Since $|\Phi|\sim N^3$, for $t=cN^{-3}$ ($c$ is a small but fixed constant),
$$
\Re \frac{e^{it\Phi}-1}{i\Phi} \gtrsim N^{-3}.
$$
Therefore
\begin{align*}
\left\| I[u_0,v_0]\right\|_{H^k} \gtrsim N^{k}\left\|\int \frac{e^{it\Phi}-1}{i\Phi}\hat{u}_0(\xi_1)\hat{v}_0(\xi_2)d\xi_1 \right\|_{L^2} \gtrsim N^{k-3}\|\hat{u}_0\ast\hat{v}_0\|_{L^2}\gtrsim N^{k-3}. 
\end{align*}
However, the continuity of $I$ as an operator from $H^k\times H^s$ into $H^s$ imposes
$$
N^{k-3}\lesssim 	\left\| I[u_0,v_0]\right\|_{H^k}  \lesssim N^{s}
$$
If $k-s>3$, this is a contradiction for large $N$.

\end{proof}

\begin{lem}
	System \eqref{nlskdv} is $C^2$-ill-posed in $H^k(\R)\times H^s(\R)$ for $(k,s)\in \mathcal{I}_2$.
\end{lem}
\begin{proof}
	For $N\gg 1$, take
	$$
	k+\frac{1}{2}<\rho<s-2+\frac{1}{2}
	$$
	and
	$$
	\hat{u}_0=\frac{1}{\jap{\xi}^\rho}\mathbbm{1}_{[0,N]},\quad \hat{v}_0=\mathbbm{1}_{[-1,1]}.
	$$ 
	The first condition on $\rho$ implies that $u_0$ is uniformly bounded in $H^k$. Assuming that the flow is $C^2$ at the origin, if the initial data is of the form $(\epsilon u_0, \epsilon v_0)$, then the $\epsilon^2$-term for $v$ is
	$$
	I[u_0,v_0]:=\partial_x\int_0^t e^{-(t-s)\partial_x^3}\left(|e^{is\partial_x^2}u_0|^2\right)ds + \partial_x\int_0^t e^{-(t-s)\partial_x^3}\left(e^{-s\partial_x^3}v_0\right)^2ds
	$$
	We consider the projection of this term for $|\xi|>2$, in which case the second term gives a zero contribution. The remaining term reduces, in frequency variables, to
	$$
	\int \frac{e^{it\Phi}-1}{i\Phi}\xi\hat{u}_0(\xi_1)\widehat{\bar{u}}_0(\xi_2)d\xi_1 =  \int \frac{e^{it\Phi}-1}{i\Phi}\xi\hat{u}_0(\xi_1)\bar{\hat{u}}_0(-\xi_2)d\xi_1,\quad \Phi=\xi^3-\xi_1^2+\xi_2^2.
	$$
	For $N/2<\xi<N$, we can ensure that $\Phi\sim N^3$ for $N$ large. Therefore, choosing $t=cN^{-3}$,
	$$\Re \frac{e^{it\Phi}-1}{i\Phi} \gtrsim N^{-3}.
	$$
	Thus
	\begin{align*}
	\left|\int \frac{e^{it\Phi}-1}{i\Phi}\xi\hat{u}_0(\xi_1)\bar{\hat{u}}_0(-\xi_2)d\xi_1\right| \gtrsim N^{-3}N\int_{0<\xi_2<N/4} \frac{1}{\jap{\xi_1}^\rho\jap{\xi_2}^\rho}d\xi_1 \gtrsim N^{-2-\rho}\int_{0<\xi_2<N/4}\frac{1}{\jap{\xi_2}^\rho}d\xi_2
	\end{align*}
	and the $L^2(\jap{\xi}^sd\xi)$ norm over the interval $[N/2,N]$ is larger than
	$$
	N^{s-2-\rho}\int_{0<\xi_2<N/4}\frac{1}{\jap{\xi_2}^\rho}d\xi_2\left(\int_{N/2}^N d\xi \right)^{1/2}= N^{s-2-\rho+\frac{1}{2}}\int_{0<\xi_2<N/4}\frac{1}{\jap{\xi_2}^\rho}d\xi_2 \gtrsim N^{s-2-\rho+\frac{1}{2}}
	$$
	(notice that $\rho>1$). By the second condition on $\rho$, this term blows up as $N\to \infty$, contradicting the boundedness of $I$.
\end{proof}

\bibliography{biblio}
\bibliographystyle{plain}

	\normalsize

\begin{center}
	{\scshape Simão Correia}\\
	{\footnotesize
		Center for Mathematical Analysis, Geometry and Dynamical Systems,\\
		Department of Mathematics,\\
		Instituto Superior T\'ecnico, Universidade de Lisboa\\
		Av. Rovisco Pais, 1049-001 Lisboa, Portugal\\
		simao.f.correia@tecnico.ulisboa.pt
	}
	\bigskip
	
	{\scshape Felipe Linares}\\
{\footnotesize
	Instituto de Matemática Pura e Aplicada,\\
	Estrada Dona Castorina, 110
	Jardim Botânico, Rio de Janeiro, Brazil\\
	linares@impa.br
}

\bigskip

	{\scshape Jorge Drumond Silva}\\
{\footnotesize
	Center for Mathematical Analysis, Geometry and Dynamical Systems,\\
	Department of Mathematics,\\
	Instituto Superior T\'ecnico, Universidade de Lisboa\\
	Av. Rovisco Pais, 1049-001 Lisboa, Portugal\\
	jsilva@math.tecnico.ulisboa.pt
}
	
\end{center}

\end{document}